\documentclass[12pt,a4paper]{article}

\usepackage[utf8]{inputenc}
\usepackage[english]{babel}

\usepackage{amsfonts,latexsym,amssymb,amsthm,amsmath,euscript,cite,xcolor} %,mathab

\usepackage[a4paper,left=2cm,right=2cm,top=2cm,bottom=1.65cm]{geometry}

\usepackage[OT4]{fontenc}
\usepackage{tikz}
\usetikzlibrary{shapes}

\usepackage{mathtools}

\usepackage{hyperref}
\usepackage{url}
\usepackage[T1]{fontenc}
\usepackage{lmodern}
\usepackage{amsfonts}
\usepackage{graphicx}
\usepackage{color}

\newtheorem{prop}{Proposition}[section]

\newtheorem{defin}[prop]{Definition}
\newtheorem{lem}[prop]{Lemma}

\newtheorem{stw}[prop]{Statement}

\numberwithin{equation}{section}
\theoremstyle{plain}
\newtheorem{theorem}{Theorem}[section]
\newtheorem{lemma}{Lemma}[section]

\newtheorem{remark}{Remark}[section]
\newtheorem{definition}{Definition}[section]
\newtheorem{example}{Example}[section]

\newtheorem*{corollary*}{Corollary}

\def\col{\mathrm{col}\,}
\def\sign{\mathrm{sign}\,}
\def\Im{\mathrm{Im}\,}

\date{}
\begin{document}

\title{Spectral analysis of self-adjoint  first order differential operators \\ with nonlocal potentials}

\author{K. D\k{e}bowska, I.~L.~Nizhnik}

\maketitle

\begin{abstract}
The paper considers the general form of self-adjoint boundary value problems for momentum operators with nonlocal potentials.  We give an analysis of the eigenvalue distribution  as zeros of the characteristic functions, for which their explicit representations are obtained in terms of nonlocal potentials. 
\end{abstract}
\section{Introduction}
The boundary values of  solutions to boundary value problems for differential equations are usually involved only in the boundary conditions.  However, as successful experience in studying important applied problems has shown, the boundary characteristics of solutions are also involved into the equations.  For problems in quantum mechanics  with interaction concentrated in the neighborhood of a single point, or on a discrete set of such points, or even in the neighborhood of a set $\Omega_0$ of a small dimension, models are effective when the interaction is considered to be centered on such a set $\Omega_0.$  From a mathematical point of view,  this idea comes down to the following: the indicated set $\Omega_0$ is attached to the boundary of the domain, the equation is considered on functions that equal to zero in the neighborhood of $\Omega_0$, and the resulting symmetric operator is extended to self-adjoint operators.  These operators describe interactions centered on $\Omega_0$.  A large number of works has been published in this direction, we will only mention the important monograph \cite{Albe_book}.  In the same area of research, \cite{AN07} proposes a method in which  the differential equation   contains boundary values
%the values  of solutions (or their characteristics)
on the set $\Omega_0$ and certain function factors.   These function factors are called nonlocal potentials in \cite{AN07, AN13}.  For a number of models with nonlocal potentials, a detailed spectral analysis \cite{AHN,AN07, AN13,CN19,DN19,Nizh09,Nizh10,Nizh11,Nizh12} has been carried out.  The very list of titles of these works indicates a fairly wide range of research.  We only note that  the studied models are the exact solvable models, which allowed to obtain a number of new efficient algorithms for solving inverse spectral problems without using analogues of the Gelfand--Levitan--Marchenko integral equations.

 Let us demonstrate the above in more detail  for the simplest self--adjoint operator $\mathfrak{D}=i\frac{d}{dx}$ in the space $L_2(-\infty,+\infty )$ of square--summable functions on the entire axis.  The domain of the operator $\mathfrak{D}$ is the Sobolev space $W_2^1( -\infty,+\infty )$ containing functions $\psi,\,\psi' \in L_2. $ We obtain an analogue of the operator $\mathfrak{D}$, which describes the point interaction at the point $x_0=0$, by considering the symmetric operator $\mathfrak{D}_{\min}\psi(x)=i\frac{d\psi(x)}{dx}$,\,\,$x\neq0$ on the set of functions
$W_2^1(\Omega),$\,$\Omega=(- \infty,+\infty)\backslash \{0\} $ satisfying the boundary conditions $\psi(-0)=0,$\,$\psi(+0)=0$.  The operator $\mathfrak{D}_{\min}^*=\mathfrak{D}_{\max}$, where the domain of $\mathfrak{D}_{\max}$ consists of the entire space $W_2^1(\Omega).$ Every self--adjoint extension $\mathfrak{D}_{\alpha}$ of the operator $\mathfrak{D}_{\min}$ is a restriction of the operator $\mathfrak{D}_{\max}$ to the  space $L_2,$ and the domain of the operator $\mathfrak{D}_{\alpha}$ consists of the set $\{\psi: \psi \in W_2^1(\Omega),\psi(+0)=e  ^{i\alpha}\psi(-0) \}$, where $\alpha$ is a real number characterizing the operator $\mathfrak{D}_{\alpha}$.  The operator $\mathfrak{D}_{\alpha}$ itself describes the point interaction at the point $x_0=0$.  Of course, the operators $\mathfrak{D}$ and $\mathfrak{D}_{\alpha}$ can be perturbed by an operator $V$ of the form $V\psi(x)=v_1(x)\psi(-0)+v_2(x)\psi(+0)$, where the functions $v_1$ and $v_2$ belong to the space $L_2(\Omega),$  without loss of the initial operators domain.  However, the operators $\mathfrak{D}+V$  and $\mathfrak{D}_{\alpha}+V$ will be self-adjoint only for $V\equiv 0.$ To preserve self-adjointness under such perturbations, it is necessary to change the domain of the considered operators  with nonlocal potentials, using nonlocal boundary conditions, when information about the solution in the entire region $\Omega$  is used in the conditions.  For example, the operator is self-adjoint $$\mathfrak{D}_{v,\alpha}=\mathfrak{D}_{\max}\psi+v(x)[\psi(-0)+\psi(+0)] $$ on the domain 
\begin{equation}\label{1.1}
 \{\psi: \psi \in W_2^1(\Omega), \psi(+0)-i\langle\psi,v\rangle_{L_2} = e^{i\alpha}[\psi(-0)  +i\langle\psi,v\rangle_{L_2}]\}. 
\end{equation}

 In this case, any complex--valued function from the space $L_2$ can serve as a nonlocal potential $v(x)$. The general case is considered in this paper. 
The detailed proofs of the self--adjointness of the momentum operators  with nonlocal potentials on the entire axis with single--point perturbations and operators on a finite interval with nonlocal potentials associated with boundary points  are given.  The resolvents of such operators for nonreal $z$ are constructed explicitly as integral operators whose kernels have a finite rank  perturbation of the Green's function of operators without nonlocal potentials.  Spectral results for such models are considered.
 It is shown  that such models are exact solvable models.
%and the resolvent of such operators is given explicitly. 
% In the case of $\alpha=0$, the explicit form of the resolvent was obtained in \cite{CN19}.
%vstavka

\section{Boundary value problem on the whole axis with  one point interaction}

Consider the simplest first-order differential operator $\mathfrak{D}\psi=i\frac{d\psi(x)}{dx}$ on the whole axis $-\infty<x<+\infty$ in the space $L_2(- \infty,+\infty)$, whose domain is the
 Sobolev space $W_2^1(-\infty,+\infty)$.
This operator is self-adjoint with an absolutely continuous spectrum on the whole axis $-\infty<\lambda<+\infty$ and with a resolvent $(\mathfrak{D}-zI)^{-1}$, which for nonreal $z$ is an integral operator:
\begin{equation}\label{2.1}
	(\mathfrak{D}-zI)^{-1}h=\int_{-\infty}^{+\infty}\,g_z(x-y)\,h(y)\,dy ,
\end{equation}
where
\begin{equation}\label{2.2}
	g_z(x)= i\sign(\Im z) \theta(-x  \Im z)e^{-izx}.
\end{equation}

Note that such a simple operator is the free momentum operator in quantum mechanics. The Green's function $g_z(x)$ of the form $(\ref{2.2})$ has the following simple properties:
\begin{equation}\label{2.3}
\begin{array}{c}
	\| g_z(\cdot)\|_{L_2}=\frac{1}{\sqrt{2|\Im z|}},\quad g_z(-0)=i\theta(\Im z),\quad g_z(+0)=-i\theta(-\Im z),\\
g_z(+0)+g_z(-0)=i\sign(\Im z),\quad g_z(+0)-g_z(-0)=-i, \quad \bar{g}_{\bar{z}}(x)=g_z(-x).
\end{array}
\end{equation}

%\subsection{}
The boundary value problem with a one point interaction at the point $x_0=0$ for the momentum operator $\mathfrak{D}$ requires consideration of the symmetric operator in the space $L_2$ of the form  $\mathfrak{D}_{\min}\psi(x)=i\frac{d\psi}{dx}$, $x\neq 0$, $x\in \Omega=(-\infty,+\infty) \setminus \{0\}$
with domain $W_2^1(\Omega)$ containing  functions $\psi(x)$ such that $\psi(-0)=0$ and $\psi(+0)=0$.
The adjoint operator of $\mathfrak{D}_{\min}$ is the operator $\mathfrak{D}_{\max}\psi(x)=i\frac{d\psi}{dx},$\,\,$x\neq 0$ and its domain is the whole space $W_2^1(\Omega).$ The symmetry of the operator $\mathfrak{D}_{\min}$ and the equality $\mathfrak{D}_{\max}=\mathfrak{D}_{\min}^{*}$ immediately follows from Green's formula:
\begin{equation}\label{2.4}
	\langle\mathfrak{D}_{\max}\psi, \varphi\rangle_{L_2}-\langle\psi,\mathfrak{D}_{\max}\varphi\rangle_{L_2}=i\{\psi(-0)\overline{\varphi(-0)}-\psi(+0)\overline{\varphi(+0)}\},\quad \psi,\varphi \in W_2^1(\Omega).
\end{equation}

\begin{theorem}\label{thm2.1}
	% \textcolor{blue}
  For the self-adjoint operator $A$ to be an extension of the symmetric operator $\mathfrak{D}_{\min}$, it is necessary and sufficient that the operator $A$ be the restriction of the operator $\mathfrak{D}_{\max}$ to the set:
	\begin{equation}\label{2.5}
		\{\psi: \psi \in W_2^1(\Omega),\,\,\psi(+0)=e^{i\alpha}\psi(-0)\},
	\end{equation}
	where $\alpha$  is a real number characterizing the operator $A$.
	Each value of $\alpha \in [0, 2\pi)$ defines a~self--adjoint extension of the operator $\mathfrak{D}_{\min}$ and it is denoted by $\mathfrak{D}_{\alpha}$. For nonreal $z$, the resolvent $(\mathfrak{D}_{\alpha}-zI)^{-1}$ is an integral operator:
	\begin{equation}\label{2.6}
		(\mathfrak{D}_{\alpha}-zI)^{-1}h(x)= \int\,\mathcal{G}_z(x,y;\alpha)\,h(y)\, dy,
	\end{equation}
	where the kernel $\mathcal{G}_z(x,y;\alpha)$ is expressed in terms of the Green's function $(\ref{2.2})$ as:
	\begin{equation}\label{2.7}
		\mathcal{G}_z(x,y,\alpha)=g_z(x-y)+g_z(x)\beta(z,\alpha)\bar{g}_{\bar{z}}(y)
	\end{equation}
	and the coefficient $\beta(z,\alpha)$ has the form
	\begin{equation}\label{2.8}
		\beta(z,\alpha)=i(1-e^{-i\alpha})\theta(\Im z)+i(e^{i\alpha}-1)\theta(-\Im z), \quad \overline{\beta}(\bar{z},\alpha)=\beta(z,\alpha).
	\end{equation}
\end{theorem}

\begin{proof}
	The symmetry of the operator $\mathfrak{D}_{\alpha}$ follows from the Green's formula $(\ref{2.4})$.  To construct the kernel $\mathcal{G}_z(x,y,\alpha)$, it suffices to solve the equation
	\begin{equation}\label{2.9}
		\psi(x)=g_z(x)i[\psi(+0)-\psi(-0)]+\int\,g_z(x-y)\,h(y)\,dy
	\end{equation}
	for any $h \in L_2$ satisfying the boundary condition $\psi(+0)=e^{i\alpha}\psi(-0).$ From the equation $(\ref{2.9})$ and the boundary conditions, we obtain a linear system for $\psi(+0)$ and $\psi(-0)$:
	\begin{equation}\label{2.10}
		\begin{array}{ll}
			\frac{1}{2}[\psi(+0)-\psi(-0)]=\frac{1}{2}\sign(\Im z)[\psi(+0)-\psi(-0)]+\langle h,g_{\bar{z}}\rangle,\\[2mm]
			\psi(+0)- e^{i\alpha}\psi(-0)=0.
		\end{array}
	\end{equation}
	If we substitute the solution of the system $(\ref{2.10})$ into the equation $(\ref{2.9})$, then we get
	\begin{equation}\label{2.11}
		\psi(x)=\int\,\mathcal{G}_z(x,y;\alpha) h(y)\,dy,
	\end{equation}
where the kernel has the form $(\ref{2.7})$. It is easy to verify that the solution $\psi$ of the form $(\ref{2.11})$, with the kernel $\mathcal{G}_z(x,y,\alpha)$, satisfies the equation $ (\mathfrak{D}_{\alpha}-zI)^{-1}\psi=h$.
\end{proof}
\begin{remark}\label{2.1}
The Green's function $\mathcal{G}_z(x,y;\alpha)$ is related to the Green's function $g_z(x)$ by the equality $(\ref{2.7})$
%(\ref{2.7}).
 Therefore, due to  (\ref{2.8})
and  (\ref{2.3}),
 the function
$$E_1(x,z)=\mathcal{G}_z(x,-0;\alpha)=g_z(x)w(z,\alpha)
$$
where $w(z,\alpha)=\theta(\Im z)+e^{i\alpha}\theta(-\Im z)$.
Therefore, $E_1(-0,z)-e^{-i\alpha}E_1(+0,z)=i.$ Due to the boundary conditions (\ref{2.5})
$$
\begin{array}{l}
  \displaystyle \frac{1}{2}\Bigl(\mathcal{G}_z(-0,y;\alpha)+e^{-i\alpha}\mathcal{G}_z(+0,y;\alpha) \Bigr)=\mathcal{G}_z(-0,y;\alpha), \\[2mm]
 \displaystyle   \frac{1}{2}\Bigl(E_1(-0,z)+e^{-i\alpha}E_1(+0,z)\Bigr)=\frac{i}{2}\sign(\Im z).
\end{array}
$$
\end{remark}
The operators $\mathfrak{D}_{\alpha}$ in Theorem $\ref{thm2.1}$ are the models of point interactions  for momentum operators at the point $x_0=0$.
\begin{remark}\label{2.2}
Self--adjoint operators $\mathfrak{D}_{\alpha}$ with $\alpha \in [0,2\pi)$ are unitary equivalent.
\end{remark}
\begin{proof}
The operator $\mathfrak{D}_0$ for $\alpha=0$ is a self-adjoint operator with the domain $W_2^1(R^1)$. This operator is given by the differential operator $\mathfrak{D}_0=i\frac{d}{dx}.$ The operator $\mathfrak{D}_{\alpha}$ is related to the operator $\mathfrak{D}_0$ by the equality $\mathfrak{D}_{\alpha}=U_{\alpha}^* \mathfrak{D}_0U_{\alpha}$ where $U_{\alpha}$ is a unitary operator in the space $ L_2(-\infty,+\infty)$.
    The operator $U_{\alpha}$ is an operator of multiplication by the function $\chi_{\alpha}(x)=1$ for $x>0$ and $\chi_{\alpha}(x)=e^{-i\alpha x}$ for $x<0$. If the function $\psi_{\alpha}(x)$ belongs to the domain of the operator $\mathfrak{D}_{\alpha}$, then $U_{\alpha} \psi_{\alpha} \in W_2^1(R^1)$ and the equality $\mathfrak{D}_0U_{\alpha}\psi_{\alpha}$ makes sense. The operator $U_{\alpha}^* $  maps the function $\mathfrak{D}_0U_{\alpha} \psi_{\alpha}$ into $\mathfrak{D}_{\alpha} \psi_{\alpha}$.
\end{proof}
Note that Remark \ref{2.2} is valid only for the first-order differential operators. For Schr\"{o}dinger operators, similar operators are different for different $\alpha$ and they describe point interactions \cite{Albe_book}.
\newpage

\section{Momentum operator with nonlocal potentials}
 Let us now extend the constructed model using nonlocal potentials.
\begin{definition}\label{def2.1}
	In the space $L_2(-\infty,+\infty)$, the momentum operator with a local interaction in the point $x_0=0$ and two nonlocal potentials $v_1(x)$ and $v_2(x)$, which are complex-valued functions from the space $ L_2$, is defined on all functions from the Sobolev space $W_2^1(\Omega)$ by the expression:
	\begin{equation}\label{2.29}
		A_{\max}\psi=\mathfrak{D}_{\max}\psi(x)+v_1(x)[\psi(-0)+\frac{i}{2}\langle\psi,v_1\rangle _{L_2}]+v_2(x)[\psi(+0)-\frac{i}{2}\langle\psi,v_2\rangle_{L_2}],\,\,x\neq 0.
	\end{equation}
	The restriction of the operator $A_{\max}$ to the domain
	\begin{equation}\label{2.210}
		\mathfrak{D}_{(v_1,v_2,\alpha)}=\{ \psi: \psi \in W_2^1(\Omega), \psi(+0)-i\langle\psi,v_2\rangle_{ L_2}=e^{i\alpha}[\psi(-0)+i\langle\psi,v_1\rangle_{L_2}] \}
	\end{equation}
	is denoted by
%$A_{(v_1,v_2,\alpha)}$
$A(v_1,v_2,\alpha)$, where $\alpha$ is a real number.
\end{definition}

\begin{lemma}\label{thm2.2}
	The momentum operators  $A(v_1,v_2,\alpha)$
%$A_{(v_1,v_2,\alpha)}$
with nonlocal potentials are symmetric operators in the space $L_2(-\infty,+\infty)$.
\end{lemma}
\begin{proof}
	 The operator $A_{\max}$\,  (\ref{2.29})
%from Definition $\ref{def2.1}$
satisfies the following Green's formula:
	\begin{equation}\label{2.211}
		\begin{array}{ll}
			\langle A_{\max}\psi, \varphi\rangle_{L_2}-\langle\psi, A_{\max}\varphi\rangle_{L_2}=i\{[\psi(-0)+i\langle\psi, v_1\rangle_{L_2}]\cdot[\varphi(-0)+i\langle\varphi,v_1\rangle_{L_2}]^*-\\ [2mm]
			-[ \psi(+0)-i\langle\psi,v_2\rangle_{L_2}] \cdot[ \varphi(+0)-i\langle\varphi,v_2\rangle_{L_2}]^{*}\}
		\end{array}
	\end{equation}
for any functions $\psi,\,\varphi \in W_2^1(\Omega)$.
	If we assume  in this Green's formula that the functions $\psi$ and $\varphi$ belong to $\mathfrak{D}_{(v_1,v_2,\alpha)}$ of $(\ref{2.210})$, then in the left--hand side of the equation $(\ref{2.211})$, we can write $A(v_1,v_2,\alpha)$
% $A_{(v_1,v_2,\alpha)}$
 instead of the operator $A_{\max}$. Using the boundary conditions of  $(\ref{2.210})$, we~obtain that the right--hand side of the equation $(\ref{2.211})$ is equal to zero. Then the operator $A(v_1,v_2,\alpha)$
 % $A_{(v_1,v_2,\alpha)}$
  is~symmetric if the domain of this operator is dense in the space $L_2(-\infty,+\infty)$.
 % $$\psi \in W_2^1(\Omega):\{ \psi(+0)-i\langle\psi,v_2\rangle_{ L_2}=0,\, \psi(-0)+i\langle\psi,v_1\rangle_{L_2}=0\} $$ is dense.
 Let us show that even the domain of  the  operator $A_{\min}:$
\begin{equation}\label{dom}
  \{\psi:\psi \in W_2^1(\Omega), \psi(+0)-i\langle\psi,v_2\rangle_{ L_2}=0,\, \psi(-0)+i\langle\psi,v_1\rangle_{L_2}=0\} \subset \mathfrak{D}_{(v_1,v_2,\alpha)}
\end{equation}
is dense in the  space $L_2(-\infty,+\infty)$. It is well--known that the set of all compactly supported continuously differentiable functions  $\{\psi_0\}\in W_2^1(-\infty,+\infty)$
such that $ \psi_0(0) = 0$
is dense  in the  space $L_2(-\infty,+\infty)$. Let us show that such functions $\psi_0(x)$ can be approximated as close as possible by functions from the domain of the operator $A_{\min}:$
$$||\psi_0-\psi_n||\rightarrow 0,\quad \,n\rightarrow \infty.
$$
 This means that the domain of the operator  $A_{\min}$  and, therefore,  the domain of the operator $A(v_1,v_2,\alpha)$  are dense in the space $L_2(-\infty,+\infty)$.

In order to prove this  we construct two functions:
\begin{equation*}
  \delta^{(+)}(x)=
  \begin{cases}
   \begin{array}{l}
                 (1-x)^2,\,\,\, x\in [0,1], \\[2mm]
                 0,\quad \quad \quad  \,\,\,  x \overline{\in} [0,1],
               \end{array}
                \end{cases}
\end{equation*}
and
\begin{equation*}
  \delta^{(-)}(x)=
    \begin{cases}
  \begin{array}{l}
                 (1+x)^2,\,\,\, x\in [-1,0], \\[2mm]
                 0,\quad \quad \quad \,\,\,  x \overline{\in} [-1,0].
               \end{array}
                \end{cases}
\end{equation*}
Then we define  functions $\delta_n^{(+)}(x)= \delta^{(+)}(nx)$ and $\delta_n^{(-)}(x)= \delta^{(-)}(nx)$,\,\,$n\in N$. Now consider the following functions:
\begin{equation*}
  \psi_n(x)=\psi_0(x)+C_n^{(+)}\delta_n^{(+)}(x)+C_n^{(-)}\delta_n^{(-)}(x).
\end{equation*}
These functions $\psi_n(x)$ belong to the space $ W_2^1(\Omega)$, and with appropriately chosen constants $C_n^{(+)}$ and $C_n^{(-)}$,  they belong to the domain  of the operator  $A_{\min}$. Thus, substituting the function $ \psi_n(x)$  into  two conditions of the formula (\ref{dom}), we get   a linear system:
\begin{equation}\label{sys}
  \begin{cases}
   \begin{array}{l}
   C_n^{(+)}-i[\langle\psi_0,v_2\rangle+  C_n^{(+)} \langle \delta_n^{(+)},v_2\rangle+C_n^{(-)} \langle\delta_n^{(-)},v_2\rangle]=0,\\[2mm]
    C_n^{(-)}+i[\langle\psi_0,v_1\rangle+  C_n^{(+)} \langle \delta_n^{(+)},v_1\rangle+C_n^{(-)} \langle\delta_n^{(-)},v_1\rangle]=0.
  \end{array}
                \end{cases}
\end{equation}
 Since $\displaystyle| \langle \delta_n^{(+)},v_j\rangle|\leq\frac{1}{\sqrt{5n}}||v_j||$ and  $\displaystyle| \langle \delta_n^{(-)},v_j\rangle|\leq\frac{1}{\sqrt{5n}}||v_j||$,\,\,$j=1,2$, then
 the system (\ref{sys}) is uniquely solvable for sufficiently large values of $n$. In addition, $\displaystyle C_n^{(+)}=i\langle\psi_0,v_2\rangle+O\Bigl(\frac{1}{\sqrt{n}}\Bigr)$  and
  $\displaystyle C_n^{(-)}=-i\langle\psi_0,v_1\rangle+O\Bigl(\frac{1}{\sqrt{n}}\Bigr)$ as $n\rightarrow\infty$.
Therefore,  $ |C_n^{(+)}|$ and  $ |C_n^{(-)}|$ are bounded by some constant $C$ as $n\rightarrow\infty$. This shows that
$$\displaystyle ||\psi_0-\psi_n||\leq \frac{C(||v_1||+||v_2||)}{\sqrt{n}}\rightarrow 0,\quad \,n\rightarrow \infty.
$$
 Thus, the domain  of the operator  $A_{\min}$ and, consequently, the domain  of the operator $A(v_1,v_2,\alpha)$ are dense in the space $L_2(-\infty,+\infty)$.
 And these operators  are symmetric in $L_2(-\infty,+\infty)$.
\end{proof}

From Definition $\ref{def2.1}$ of the operators $A(v_1,v_2,\alpha)$
% $A_{(v_1,v_2,\alpha)}$
 it follows that two operators with different nonlocal potentials can have the same domains and even coincide.

\begin{theorem}\label{thm2.3}
	A necessary and sufficient conditions for the operators
%$A_{(v_1,v_2,\alpha)}$
 $A(v_1,v_2,\alpha)$
 and
% $A_{(\hat{v}_1,\hat{v}_2,\alpha)}$
  $A(\hat{v}_1,\hat{v}_2,\alpha)$
to  have the same domains are the following:
	\begin{equation}\label{2.12}
		v_1(x)+e^{i\alpha}v_2(x)=\hat{v}_1(x)+e^{i\alpha}\hat{v}_2(x) \overset{def}=v (x).
	\end{equation}
	The difference of two operators with the same domains
	\begin{equation}\label{2.13}
		%A_{(v_1,v_2,\alpha)} -A_{(v, 0,\alpha)}=K_{(v_1,v_2,\alpha)}
A(v_1,v_2,\alpha) -A(v, 0,\alpha)=K(v_1,v_2,\alpha)
	\end{equation}
	is an integral operator of rank at most $2$
	\begin{equation}\label{2.14}
		%K_{(v_1,v_2,\alpha)}\psi=\frac{i}{2}e^{i\alpha}v_2(x)\langle\psi,v_1\rangle-\frac{i}{2}e^ {-i\alpha}v_1(x)\langle\psi,v_2\rangle.
K(v_1,v_2,\alpha)\psi=\frac{i}{2}e^{i\alpha}v_2(x)\langle\psi,v_1\rangle-\frac{i}{2}e^ {-i\alpha}v_1(x)\langle\psi,v_2\rangle.
	\end{equation}
	The difference of two operators
 %$A_{(v_1,v_2,\alpha)} -A_{(\hat{v}_1,\hat{v}_2,\alpha)}$
$A(v_1,v_2,\alpha) -A(\hat{v}_1,\hat{v}_2,\alpha)$
with nonlocal potentials under the condition $(\ref{2.12})$ is an integral operator of rank at most $4$. These integral operators are bounded and symmetric.
\end{theorem}
\begin{proof}
	The domain of the operators
%$A_{(v_1,v_2,\alpha)}$
$A(v_1,v_2,\alpha)$
 is determined by the boundary condition (\ref{2.210}) and can be rewritten as follows:
	\begin{equation}\label{2.14}
		\psi(+0)-e^{i\alpha}\psi(-0)=i[e^{i\alpha}\langle\psi,v_1\rangle+\langle\psi,v_2\rangle]
	\end{equation}
% \psi(+0)-i\langle\psi,v_2\rangle_{ L_2}=e^{i\alpha}[\psi(-0)+i\langle\psi,v_1\rangle_{L_2}]
	and  respectively  for the operator
% $A_{(\hat{v}_1,\hat{v}_2,\alpha)}$
   $A(\hat{v}_1,\hat{v}_2,\alpha)$.
 %  The equivalence of such boundary conditions reduces to the identity $(\ref{2.12})$.
   Let us determine the value $\psi(+0)$ of the boundary condition $(\ref{2.14})$ and substitute it into the definition of the operator
 % $A_{(v_1,v_2,\alpha)}$.
   $A(v_1,v_2,\alpha)$.
  We obtain
	\begin{equation}
	%	A_{(v_1,v_2,\alpha)}\psi=A_{(v, 0,\alpha)}\psi+K_{(v_1,v_2,\alpha)}\psi,
A(v_1,v_2,\alpha)\psi=A(v, 0,\alpha)\psi+K(v_1,v_2,\alpha)\psi,
	\end{equation}
	where
%$K_{(v_1,v_2,\alpha)}$
$K(v_1,v_2,\alpha)$
 is a bounded symmetric operator of rank at most $2$. Since this is also true for the operator
 %$A_{(\hat{v}_1,\hat{v}_2,\alpha )}$,
  $A(\hat{v}_1,\hat{v}_2,\alpha )$,
  then
 % $A_{(v_1,v_2,\alpha)}- A_{(\hat{v}_1,\hat{v}_2,\alpha)}$
   $A(v_1,v_2,\alpha)- A(\hat{v}_1,\hat{v}_2,\alpha)$
   is a difference of two symmetric operators. Therefore, this is a bounded symmetric operator of rank at most~$4$.
\end{proof}
\begin{remark}\label{rem2.1}
	The fact that operators with nonlocal potentials and the same domain  can differ by a bounded symmetric operator shows that the proof of the self-adjointness of an arbitrary operator
	%$A_{(v_1,v_2,\alpha)}$
$A(v_1,v_2,\alpha)$
can be reduced to proving the self-adjointness of the operators
%$A_{(v,0,\alpha)}$
$A(v,0,\alpha)$
 or
% $A_{(\frac{1}{2}v,\frac{ e^{-i\alpha}}{2}v, \alpha)}$,
  $A(\frac{1}{2}v,\frac{ e^{-i\alpha}}{2}v, \alpha)$,
 where $v=v_1+e^{i\alpha}v_2$ and the operators
% $A_{(v,0,\alpha)}$,\, $A_{(\frac{1}{2}v,\frac{e^{-i\alpha}}{2}v, \alpha)}$
  $A(v,0,\alpha)$,\, $A(\frac{1}{2}v,\frac{e^{-i\alpha}}{2}v, \alpha)$
  coincide.
\end{remark}

Indeed, due to Theorem $\ref{thm2.3}$ we have
%$$A_{(\frac{1}{2}v,\frac{e^{-i\alpha}}{2}v, \alpha)}-A_{(v,0,\alpha)}=K_{ (\frac{1}{2}v,\frac{e^{-i\alpha}}{2}v, \alpha)}\equiv 0.$$
$$A(\frac{1}{2}v,\frac{e^{-i\alpha}}{2}v, \alpha)-A(v,0,\alpha)=K(\frac{1}{2}v,\frac{e^{-i\alpha}}{2}v, \alpha)\equiv 0.$$
We will denote the operator
%$A_{(\frac{1}{2}v,\frac{e^{-i\alpha}}{2}v, \alpha)}$ by
$A(\frac{1}{2}v,\frac{e^{-i\alpha}}{2}v, \alpha)$ by
%$A_{(v,\alpha)}. $
$A(v,\alpha). $
Its action is expressed in  the form
\begin{equation}\label{2.16}
	%A_{(v,\alpha)}\psi= \mathfrak{D}_{\max}\psi+v(x)\left[\frac{1}{2}\psi(-0)+\frac{e^ {-i\alpha}}{2} \psi(+0)\right]
A(v,\alpha)\psi= \mathfrak{D}_{\max}\psi+v(x)\left[\frac{1}{2}\psi(-0)+\frac{e^ {-i\alpha}}{2} \psi(+0)\right]
\end{equation}
with the boundary condition
\begin{equation}\label{2.17}
	i\psi(-0)-ie^{-i\alpha}\psi(+0)=\langle\psi,v\rangle.
\end{equation}

\begin{theorem}\label{thm2.4}
	For any potential $v \in L_2$ and a real number $\alpha$, the operator
 %$A_{(v,\alpha)}$
  $A(v,\alpha)$
  of the form $(\ref{2.16})$--$(\ref{2.17})$ is self-adjoint in the space $ L_2$. For nonreal $z$, the resolvent
 % $(A_{(v,\alpha)}-zI)^{-1}$
   $(A(v,\alpha)-zI)^{-1}$
   is an~integral operator with the kernel $\mathcal{G}_{z}(x,y;v, \alpha )$
	\begin{equation}\label{2.18}
	%	(A_{(v,\alpha)}-zI)^{-1}h(z)=\int\,\mathcal{G}_{z}(x,y;v, \alpha)h(y) \,dy.
(A(v,\alpha)-zI)^{-1}h(z)=\int\,\mathcal{G}_{z}(x,y;v, \alpha)h(y) \,dy.
	\end{equation}
The kernel $\mathcal{G}_{z}(x,y;v, \alpha)$ is a rank at most $2$ perturbation of the kernel  $\mathcal{G}_z(x,y;\alpha)$ of the free momentum operator  in (\ref{2.7}) (see Theorem \ref{thm2.1}):
\begin{equation}\label{2.19}
	%\mathcal{G}_{z}(x,y;v, \alpha)=g_z(x-y)+\frac{1}{\gamma(z)}\sum_{j,k=1}^{2}\,e_j(x,z)a_{j,k}(z)\bar{e}_k(y,\bar{z}),
\displaystyle \mathcal{G}_z(x,y;v,\alpha)=\mathcal{G}_z(x,y;\alpha)+\frac{1}{\gamma(z)}\sum\limits_{j ,k=1}^{2}\,E_j(x,z)\gamma_{jk}(z)\bar{E}_k(y,\bar{z}),
\end{equation}
 where $E_1(x,z)=\mathcal{G}_z(x,-0;\alpha),$\,\,$E_2(x,z)=\int\, \mathcal{G}_z(x,y ;\alpha)v(y)\,dy.$ The parameters $\gamma_{jk}(z)$ and $\gamma(z)$ are defined by the formulas:
\begin{equation}\label{2.20}
	\begin{array}{ll}
	%	\gamma(z)=i\big\{[\theta(Imz)+\frac{1}{2}e^{-i\alpha}(v,g_{\overline{z}})][1+(g_z,v)-\frac{i}{2}(e_2,v)]+\\
	%	+[\theta(-Imz)+\frac{1}{2}(v,g_{\overline{z}})][e^{-i\alpha}+(g_z,v)+\frac{i}{2}e^{-i\alpha}(e_2,v)]\big\}   ,\\[2mm]
	%	a_{11}(z)=(e^{-i\alpha}-1)+\frac{i}{2}(e^{-i\alpha}+1)(e_2,v)	,\\[2mm]
	%	a_{12}(z)=-i-\frac{i}{2}(e^{-i\alpha}+1)(v,g_{\overline{z}})	,\\[2mm]
	%%	a_{21}(z)=-ie^{-i\alpha}-\frac{i}{2}(e^{-i\alpha}+1)(g_z,v)	,\\[2mm]
	%a_{22}(z)=-\frac{1}{2}\theta(Imz)+\frac{1}{2}e^{-i\alpha}\theta(-Imz)	,\\[2mm]
	% e_1(x,z)=g_z(x),\, e_2(x,z)=\int\, g_z(x-y)v(y)\,dy.
\displaystyle \gamma_{11}(z)=-\langle E_2(\cdot,z),v(\cdot)\rangle_{L_2}=-\int\int\,\mathcal{G}_z(x,y;\alpha )v(y)\overline{v(x)}\,dy\,dx, \\[2mm]
\displaystyle \gamma_{12}(z)=1+\langle v(\cdot),E_1(\cdot,\overline{z})\rangle_{L_2}=1+\int \, \overline{\mathcal{G} _{\bar{z}}}(-0,y;\alpha)v(y)\,dy, \\[2mm]
 \displaystyle\gamma_{21}(z)=1+\langle E_1(\cdot,z),v(\cdot)\rangle_{L_2}=1+\int\, \mathcal{G}_z(x,-0;\alpha)\overline{v(x)}\,dx, \\[2mm]
\displaystyle \gamma_{22}(z)=\displaystyle -\frac{i}{2} \sign(\Im z),\\[2mm]
\displaystyle \gamma(z)=\gamma_{11}(z) \gamma_{22}(z)-\gamma_{12}(z) \gamma_{21}(z)=\\[2mm]
\displaystyle =\displaystyle \frac{i}{2} \sign (\Im z) \cdot\langle E_2(\cdot,z ),v(\cdot)\rangle_{L_2}-(1+\langle E_1(\cdot,z),v(\cdot)\rangle_{L_2})(1+\langle v(\cdot),E_1(\cdot ,\overline{z})\rangle_{L_2}).
	\end{array}
\end{equation}
The matrix $\Gamma(z)=\|\gamma_{jk}(z)\|$ has the property $\Gamma^*(\overline{z})=\Gamma(z),$ and $\gamma(z) =\det \Gamma$ has the property
$\overline{\gamma}(\overline{z})=\gamma(z)\neq 0$.
\end{theorem}
\begin{proof}
	The operator $A(v,\alpha)$
%$A_{(v,\alpha)}$
 is symmetric in the space $L_2$ according to Theorem $\ref{thm2.2}$. Therefore, in order to prove its self--adjointness, it is sufficient to show that for any $h$ and some $z_1$,\,\,$z_2$ for which
 $\Im z_1>0,$\,\,$\Im z_2<0$ there exists a solution of the problem
\begin{equation}\label{2.21}
	%(A_{(v,\alpha)}-z_j I)^{-1}\psi(x)=h,\quad j=1,2.
(A(v,\alpha)-z_j I)^{-1}\psi(x)=h,\quad j=1,2.
\end{equation}
This problem is equivalent to the problem for finding a solution of the equation:
	$$\mathfrak{D}_{\max}\psi -z_j \psi =h-\frac{1}{2}v(x)(\psi(-0)+e^{-i\alpha}\psi(+0)),\quad j=1,2,$$
satisfying the boundary condition $(\ref{2.17})$. This problem can be reduced to the problem
\begin{equation}\label{2.22}
%	\psi(x)=g_z(x)i[\psi(+0)-\psi(-0)]-e_2(x,z)\frac{1}{2}[\psi(-0)+e ^{-i\alpha}\psi(+0)]+\int\,g_z(x-y)h(y)\,dy
\psi(x)=-\mathcal{G}_z(x,-0;\alpha)\cdot\langle\psi,v\rangle+\int\,\mathcal{G}_z(x,y;\alpha)\left[h(y)-v(y)\frac{1}{2}(\psi(-0)+e^{-i\alpha}\psi(+0))\right]\,dy
\end{equation}
with the boundary condition $(\ref{2.17})$.
Indeed, if  the equation (\ref{2.22}) has a solution $\psi(x)\in L_2$, then it satisfies the equation
$$\left(i\frac{d}{dx}-zI\right)\psi(x)=h(x)-v(x)\frac{1}{2}\left(\psi(-0 )+e^{-i\alpha}\psi(+0)\right)
$$
and the boundary conditions
(\ref{2.17}).
 This follows from the properties of the Green's function $\mathcal{G}_z(x,y,\alpha)$.
A function
$$E_1(x,z)=\mathcal{G}_z(x,-0;\alpha)=g_z(x)w(z,\alpha)
$$
satisfies an equality
$$E_1(-0,z)-e^{-i\alpha} E_1(+0,z)=i
$$
in view of properties (\ref{2.3}).
The function representing the integrals in (\ref{2.22}) satisfies homogeneous conditions due to the boundary conditions
(\ref{2.5}).
Therefore, the solution $\psi(x)$  of the equation (\ref{2.22}) satisfies the condition
$$\psi(-0)-e^{-i\alpha}\psi(+0)=-i\langle\psi,v\rangle,
$$
that is the boundary conditions for elements of the domain  of the operator $A(v,\alpha).$
The equation (\ref{2.22}) can be reduced to a linear system of equations for two numbers $\langle\psi,v\rangle$ and $\psi_s=\frac{1}{2}\Bigl(\psi(-0)+e^ {-i\alpha}\psi(+0)\Bigr)$.
In order to do this, one need to represent the quantities $\psi_s$ and $\langle\psi,v\rangle$ on the left-hand side.
Since any function of the form $\displaystyle \varphi(x)=\int\,\mathcal{G}_z(x,y,\alpha)f(y)\,dy$,\,\,$f\in L_2$
has the property
$$\varphi_s=\frac{1}{2}\Bigl(\varphi(-0)+e^{-i\alpha}\varphi(+0)\Bigr)=\int\,\mathcal{G}_z (-0,y;\alpha)f(y)\,dy$$
by virtue of boundary properties for the Green's function $\mathcal{G}_z$.
The function $E_1(x,z)$ satisfies the condition
$$\frac{1}{2}\Bigl(E_1(-0,z)+e^{-i\alpha}E_1(+0,\alpha)\Bigr)=\frac{i}{2}\sign (\Im z).
$$
Then  we obtain  from (\ref{2.22}) the following equality:
\begin{equation}\label{2.27d}
\psi_s=-\frac{i}{2}\sign (\Im z)\langle \psi, v\rangle+\int\,\mathcal{G}(-0,y;\alpha)(h(y)-v( y)\psi_s)\,ds.
\end{equation}
If the equation (\ref{2.22}) is  scalar multiplied    on the right-hand side by the function $v$ in $L_2$, then we obtain
\begin{equation}\label{2.28d}
\langle\psi, v\rangle=-\langle E_1(\cdot,z),v\rangle\langle\psi, v\rangle+\langle h,E_2(\cdot,\bar{z})\rangle-\langle v,E_2(\cdot, x)\rangle\psi_s.
\end{equation}
The equations $(\ref{2.27d})$--(\ref{2.28d}) represent a linear system of the form
\begin{equation}\label{2.29d}
  \begin{pmatrix} -\gamma_{22}&\gamma_{12} \\ \gamma_{21}&-\gamma_{11} \end{pmatrix} \begin{pmatrix} \langle\psi,v\rangle\\ \psi_s\end{pmatrix}=
    \begin{pmatrix} \langle h,E_1\rangle\\ \langle h, E_2\rangle \end{pmatrix}.
\end{equation}
A matrix determinant
$$ \begin{vmatrix}- \gamma_{22}&\gamma_{12} \\ \gamma_{21}&-\gamma_{11} \end{vmatrix}=\gamma(z)= \gamma_{11} \gamma_{22}-\gamma_{12} \gamma_{21}
$$
 for $|\Im z|\rightarrow \infty,$\,\,$\gamma(z)\rightarrow -1.$  Therefore, there are numbers $z_1$ and $z_2$ such that
$\Im z_1>0$, \,\,$\Im z_2<0$ with sufficiently large $|\Im z_k|$ that the system (\ref{2.29d}) has a unique solution. Therefore, substituting the solution from the system (\ref{2.29d}) into (\ref{2.22}), we obtain that the equation (\ref{2.22})  has a solution for any $h\in L_2$ and  for such $z=z_k,\,\,k=1,2$. This means that the symmetric operator $A(v,\alpha)$ has a deficiency index (0,0), and, therefore, this operator is self-adjoint in the space $L_2$. Since the solution to the system (\ref{2.29d}) can be represented as
\begin{equation}\label{2.30d}
  \begin{pmatrix}\langle\psi,v\rangle\\\psi_s \end{pmatrix}=\frac{1}{\gamma(z)}
  \begin{pmatrix} -\gamma_{11}&-\gamma_{12} \\ -\gamma_{21}&-\gamma_{22} \end{pmatrix} \begin{pmatrix}\langle h,E_1(\cdot, \bar{z})\rangle\\ \langle h,E_2(\cdot,\bar{z})\rangle \end{pmatrix},
\end{equation}
then substituting these solutions into (\ref{2.22}), we obtain a representation of the solution $\psi(x)$ of
equation (\ref{2.22}) for $h$ in the form
 (\ref{2.18})
where the Green's function $\mathcal{G}_z(x,y,v,\alpha)$ has the form (\ref{2.19}) with the specified parameter values in (\ref{2.21}).
\end{proof}
Since the Green's function $\mathcal{G}_z(x,y,\alpha)$ is explicitly expressed through the Green's function $g_z(x)$ according to formula
(\ref{2.7}),
then the Green's function $\mathcal{G}_z (x,y,v,\alpha)$
is explicitly expressed through $\mathcal{G}_z(x,y,\alpha)$ according to Theorem  \ref{thm2.4}.
Therefore, this function can be expressed explicitly in terms of $g_z(x)$.
%\begin{theorem}\label{thm2.5}
%	For any $v_1,v_2 \in L_2$ and real $\alpha$, the operator $A_{(v_1,v_2,\alpha)}$ is self-adjoint in the space $L_2$. For nonreal $z$, the resolvent $(A_{(v_1,v_2,\alpha)}-zI)^{-1}$ of %$A_{(v_1,v_2,\alpha)}$ is an integral operator with the kernel
%	\begin{equation}\label{2.25}
%		\mathcal{G}_E(x,y;v,\alpha)=
%	\end{equation}
%	where $e_0(x,z)=g_z(x),$\,\,$e_j(x,z)=\int\,g_z(x-y)v_j(y)\,dy$,\,\,$j= 1,2$, and the functions $\gamma(z)$ and $a_{jk}$ are explicitly expressed in terms of $(v_1,v_2,\alpha)$.
%\end{theorem}
%\begin{proof}
%	The operator $A_{(v_1,v_2,\alpha)}$ is self-adjoint, since due to Theorem $\ref{2.3}$ the operator $A_{(v,\alpha)}$ is~self-adjoint, where $v=v_1+e^{ i\alpha }v_2$. Moreover, %$A_{(v_1,v_2,\alpha)}=A_{(v,\alpha)}+K$, where $K$ is a bounded self-adjoint operator of finite rank. Hence, we obtain the explicity form of the resolvent of the operator $A_{(v_1,v_2,\alpha)}$.
%\end{proof}
\begin{theorem}\label{thm2.5}
%A Green's function $\mathcal{G}_z(x,y,v,\alpha)$, as the kernel of the integral operator of the resolvent $(A(v,\alpha)-zI)^{-1}$ for complex $z$ and a nonlocal potential $v(x)$ is expressed as a rank 2 %perturbation of the Green's function
%$g_z(x)$ of the free momentum operator in the form
%\begin{equation}\label{2.31d}
%\mathcal{G}_z(x,y,v,\alpha)=g_z(x-y)+\frac{1}{\gamma(z)}\sum\limits_{j,k=1}^2\, e_j (x,z)a_{jk}(z)\overline{e_k}(y,\overline{z}).
%\end{equation}
%The explicit form of the parameters $a_{jk}$ and $\gamma(z)$ are following:
%\begin{equation}\label{2.32d}
%\begin{array}{l}
  % a_{11}(z)=\gamma_{11}(z)+ \\
    % a_{12}(z)=\gamma_{12}(z)=1+<v,E_1(\cdot,\bar{z})>=1+\bar{w}(\bar{z})< v,e_1(\cdot,\bar{z})>, \\
     %a_{21}(z)=\gamma_{21}(z)=1+<E_1(\cdot, z), v>=1+w(z)<e_1(\cdot,z), v>, \\
   %\displaystyle  a_{22}(z)=\gamma_{22}(z)=-\frac{i}{2} \sign (\Im z),\\
   %\gamma(z)=\gamma_{11}\gamma_{21}-\gamma_{12}\gamma_{22}= a_{11} a_{22}- a_{12} a_{21}=
%\end{array}
%\end{equation}
For any potential $v \in L_2$ and a real number $\alpha$, an operator
% $A_{(v,\alpha)}$
  $A(v,\alpha)$
 of the form
\begin{equation}\label{2.16d}
	%A_{(v,\alpha)}\psi= i \frac{d\psi(x)}{dx}+v(x)\left[\frac{1}{2}\psi(-0)+\frac{e^ {-i\alpha}}{2} \psi(+0)\right],\quad x\neq 0
	A(v,\alpha)\psi= i \frac{d\psi(x)}{dx}+v(x)\left[\frac{1}{2}\psi(-0)+\frac{e^ {-i\alpha}}{2} \psi(+0)\right],\quad x\neq 0
\end{equation}
with the boundary condition
\begin{equation}\label{2.17d}
	i\psi(-0)-ie^{-i\alpha}\psi(+0)=\langle\psi,v\rangle
\end{equation}
is self-adjoint in the space $ L_2$. For nonreal $z$, the resolvent
%$(A_{(v,\alpha)}-zI)^{-1}$
$(A(v,\alpha)-zI)^{-1}$
 is an~integral operator with the kernel $\mathcal{G}_{z}(x,y;v, \alpha )$
	\begin{equation}\label{2.18d}
	%	(A_{(v,\alpha)}-zI)^{-1}h(z)=\int\,\mathcal{G}_{z}(x,y;v, \alpha)h(y) \,dy.
(A(v,\alpha)-zI)^{-1}h(z)=\int\,\mathcal{G}_{z}(x,y;v, \alpha)h(y) \,dy.
	\end{equation}
The kernel $\mathcal{G}_{z}(x,y;v, \alpha)$ is a perturbation of the kernel $g_z(x-y)$ of the free momentum operator $(\ref{2.2})$ of rank at most $2$:
\begin{equation}\label{2.19d}
	\mathcal{G}_{z}(x,y;v, \alpha)=g_z(x-y)-\frac{1}{\gamma(z)}\sum_{j,k=1}^{2}\,e_j(x,z)\gamma_{j,k}(z)\bar{e}_k(y,\bar{z}),
\end{equation}
where
\begin{equation}\label{2.20d}
	\begin{array}{ll}
		\gamma(z)=\gamma_{11}(z)\gamma_{22}(z)-\gamma_{12}(z)\gamma_{21}(z)  ,\\[2mm]
\displaystyle		\gamma_{11}(z)=\langle e_2,v\rangle-\frac{2i(1-e^{i\alpha})}{1+e^{i\alpha}},\\[2mm]
\displaystyle		\gamma_{12}(z)=-\frac{2}{1+e^{-i\alpha}}-\langle v,g_{\overline{z}}\rangle	,\\[2mm]
\displaystyle		\gamma_{21}(z)=-\frac{2}{1+e^{i\alpha}}-\langle g_z,v\rangle	,\\[2mm]
		\gamma_{22}(z)=\frac{i}{1+e^{-i\alpha}}\left[\theta(\Im z)-e^{-i\alpha}\theta(-\Im z)\right]		,\\[2mm]
		e_1(x,z)=g_z(x),\, e_2(x,z)=\int\, g_z(x-y)v(y)\,dy.
	\end{array}
\end{equation}
\end{theorem}
\begin{proof}
%Let us present the relation between the functions $E_1(x,z)$ and $E_2(x,z)$ with $e_1(x,z)=g_z(x)$ and $e_2(x,z)=\int\,\mathcal{ G}_z(x,y,\alpha)v(y)\,dy:$
%$$
%\begin{array}{l}
%  E_1(x,z)=w(z)e_1(x,z),\\
%\displaystyle  E_2(x,z)=\int\,\mathcal{G}_z(x,y,\alpha)v(y)\,dy=e_2(x,z)+e_1(x,z)\beta(z )<v, e_1(\cdot,\overline{z})>.
%\end{array}
%$$
%Then, substituting these values into the form (\ref{2.23}), we have (\ref{2.31}) with the stated above parameter relations (\ref{2.32}). Note that the matrix $\mathfrak{A}=||a_{jk}||_{j,k=1}^{2}$ has the %property $\mathfrak{A}^*(\overline{z}) =\mathfrak{A}(z)$.
The operator $A(v,\alpha)$
% $A_{(v,\alpha)}$
 is self-adjoint in the space $L_2$ according to Theorem \ref{thm2.4}. Then, it is sufficient to show that for any $h\in L_2$ and some nonreal $z$
 there exists a solution of the problem
\begin{equation}\label{2.21d}
	%(A_{(v,\alpha)}-z I)^{-1}\psi(x)=h.
(A(v,\alpha)-z I)^{-1}\psi(x)=h.
\end{equation}
%This problem is equivalent to the problem for finding a solution of the equation
%	$$\mathfrak{D}_{\max}\psi -z_j \psi =h-\frac{1}{2}v(x)(\psi(-0)+e^{-i\alpha}\psi(+0)),\quad j=1,2,$$
%satisfying the boundary condition $(\ref{2.17})$.
This problem can be reduced to the problem
\begin{equation}\label{2.22d}
	\psi(x)=-g_z(x)\psi_g-e_2(x,z)\psi_s+\int\,g_z(x-y)h(y)\,dy
\end{equation}
with the boundary condition $(\ref{2.17d})$, where $\psi_g=i\left[\psi(-0)-\psi(+0)\right]$, $\psi_s=\frac{1}{2}\left[\psi(-0)+e^{-i\alpha}\psi(+0)\right]$. This is equivalent to solving a linear system with respect to $\psi_s$ and $\psi_g$. If we substitute $\psi_s$ from $(\ref{2.22d})$, and  the solution $\psi(x)$ from $(\ref{2.22d})$ to the boundary condition $(\ref{2.17d})$, we get a linear system
\begin{equation}\label{2.23d}
		\begin{pmatrix}
			a_{11}(z) & a_{12}(z)\\
			a_{21}(z) &a_{22}(z)
		\end{pmatrix}\begin{pmatrix} \psi_g\\\psi_s \end{pmatrix}=\begin{pmatrix}\langle h,  g_{\bar{z}}\rangle\\\langle h,e_2(\cdot,\bar{z})\rangle \end{pmatrix},
\end{equation}
where
\begin{equation*}
	\begin{array}{ll}
\displaystyle		a_{11}(z)=\frac{i}{1+e^{-i\alpha}}\left[\theta(\Im z)-e^{-i\alpha}\theta(-\Im z)\right]	,\\[2mm]
\displaystyle		a_{12}(z)=\frac{2}{1+e^{-i\alpha}}+\langle v,g_{\overline{z}}\rangle	,\\[2mm]
\displaystyle		a_{21}(z)=\frac{2}{1+e^{i\alpha}}+\langle g_z,v\rangle	,\\[2mm]
\displaystyle		a_{22}(z)=\langle e_2,v\rangle-\frac{2i(1-e^{i\alpha})}{1+e^{i\alpha}}.
	\end{array}
\end{equation*}
The determinant of the system $(\ref{2.23d})$ is equal to
\begin{align*}
	\gamma(z)&=\left[\frac{i}{1+e^{-i\alpha}}\left(\theta(\Im z)-e^{-i\alpha}\theta(-\Im z)\right)\right]	\left[\langle e_2,v\rangle-\frac{2i(1-e^{i\alpha})}{1+e^{i\alpha}}\right]\\
	&-\left[\frac{2}{1+e^{-i\alpha}}+\langle v,g_{\overline{z}}\rangle\right]\left[\frac{2}{1+e^{i\alpha}}+\langle g_z,v\rangle\right].
\end{align*}
It is easy to see that if $|\Im z_j |$,\, $j=1,2$ are large enough, then $\gamma(z_j) \neq 0$. Therefore, the system $(\ref{2.23d})$ has a unique solution. Substituting the obtained values of $\psi_g$ and $\psi_s$ into $(\ref{2.22d})$, we get a representation of the solution $\psi(x)$ in terms of $h$ as follows:
\begin{equation}\label{2.24d}
	\psi(x)=\int\,\mathcal{G}_z(x,y;v,\alpha)h(y)\,dy,
\end{equation}
where $\mathcal{G}_z(x,y;v,\alpha)$ has the form $(\ref{2.19d})$--$(\ref{2.20d})$ with \begin{equation}
	\begin{array}{ll}
		\gamma_{11}(z)=a_{22}(z),\\[2mm]
		\gamma_{12}(z)=-a_{12}(z)	,\\[2mm]
		\gamma_{21}(z)=-a_{21}(z)	,\\[2mm]
		\gamma_{22}(z)=a_{11}(z)	.\\[2mm]
		%e_1(x,z)=g_z(x),\, e_2(x,z)=\int\, g_z(x-y)v(y)\,dy.
	\end{array}
\end{equation} The matrix $\Gamma(z)=\|\gamma_{jk}(z)\|$ has a property $\Gamma^*(\overline{z})=\Gamma(z)$ and
$$\gamma(z)=a_{11}(z)a_{22}(z)-a_{21}(z)a_{12}(z)=%\gamma_{22}(z)\gamma_{11}(z)-(-\gamma_{21}(z))(-\gamma_{12}(z))=
\gamma_{11}(z)\gamma_{22}(z)-\gamma_{12}(z)\gamma_{21}(z). $$
\end{proof}
\begin{remark}\label{3.1}
Since   the operator $A(v_1,v_2,\alpha)$
% $A_{(v_1,v_2,\alpha)}$
differs from the operator
% $A_{(v,\alpha)}$
  $A(v,\alpha)$
  by a symmetric operator of rank at most 2 according to  Lemma \ref{thm2.3}, then its resolvent differs from the resolvent $\mathcal{G}_z(x ,y,v,\alpha)$ by an operator of rank 2.

Recall that if the self-adjoint operator $A$ is perturbed   in a Hilbert space $H$ by the bounded operator
$$V\psi=v_1\langle\psi,v_2\rangle+v_2\langle\psi,v_1\rangle, \quad v_1,v_2 \in H,
$$
where $v_1, v_2 \in H,$
  then the resolvent
$$
[(A+V-zI)^{-1}]h=R_z h-( R_zv_1,R_zv_2)\cdot\Gamma^{-1}(z)\cdot \begin{pmatrix} \langle h, v_1\rangle\\ \langle h , v_2\rangle\end{pmatrix}
$$
where $R_z=(A-zI)^{-1}$ and a numeric matrix has the following form:
$$
\Gamma(z)=\begin{pmatrix} \langle R_zv_1, v_1\rangle& 1+\langle R_zv_1, v_2\rangle \\ 1+\langle R_zv_2, v_1\rangle & \langle R_zv_2, v_2\rangle\end{pmatrix}.$$
\end{remark}

An explicit formula for the resolvent of the operator $A(v,\alpha)$
%$A_{(v,\alpha)}$
 from Theorem \ref{thm2.4}, Theorem  \ref{thm2.5}  and
 Remark \ref{rem2.1} allows one to get a representation of the Green's function of operators
%$A_{(v_1,v_2,\alpha)}$
$A(v_1,v_2,\alpha)$
in the form of rank 4 perturbations of the Green's function $g_z(x)$ of the free momentum operator and the operator with one potential.
However, it is possible to  obtain directly a representation of the resolvent by a rank 3 perturbation.
\begin{theorem}\label{thm3.5}
   For potentials $v_1,v_2 \in L_2$ and a real number $\alpha \in [0,2\pi)$ describing the boundary conditions (\ref{2.210}), the operator $A(v_1,v_2,\alpha)$
% $A_{(v_1,v_2,\alpha)}$
    is self-adjoint in the space $ L_2$.
For nonreal $z$,  its  resolvent $(A(v_1,v_2,\alpha)-zI)^{-1}$
%$(A_{(v_1,v_2,\alpha)}-zI)^{-1}$
 is an~integral operator with the kernel $\mathcal{G}_{z}(x,y;v_1,v_2, \alpha )$
	\begin{equation}\label{3.30d}
		(A(v_1,v_2,\alpha)-zI)^{-1}h(z)=\int\,\mathcal{G}_{z}(x,y;v_1,v_2, \alpha)h(y) \,dy.
%(A_{(v_1,v_2,\alpha)}-zI)^{-1}h(z)=\int\,\mathcal{G}_{z}(x,y;v_1,v_2, \alpha)h(y) \,dy.
	\end{equation}
The kernel $\mathcal{G}_{z}(x,y;v, \alpha)$ has the following representation:
\begin{equation}\label{3.31d}
\displaystyle \mathcal{G}_z(x,y;v_1,v_2,\alpha)=\mathcal{G}_z(x,y,\alpha)-\frac{1}{\gamma(z)}\sum\limits_{j,k=1}^{2}\mathcal{E}_j(x,z)\gamma_{jk}\bar{\mathcal{E}}_k(y,\bar{z}),
\end{equation}
where
\begin{equation}\label{3.32d}
\begin{array}{ll}
 E_0(x,z)=\mathcal{G}_z(x,-0;\alpha),&\quad E_k(x,z)=\int\, \mathcal{G}_z(x,y;\alpha)v_k(y)\,dy,\quad k=1,2,\\[2mm]
\mathcal{E}_1(x,z)=E_1+2iE_0, &\quad \mathcal{E}_2(x,z)=E_2-2ie^{-i\alpha}E_0,
\end{array}
 \end{equation}
and
\begin{equation}
\begin{array}{l}
\gamma(z)= \gamma_{11}(z) \gamma_{22}(z)-\gamma_{12}(z) \gamma_{21}(z),\\[2mm]
	\gamma_{11}(z)=2i\left(\sign(\Im z)+e^{i\alpha}\langle v_2,E_0(\cdot,{\bar{z}})\rangle-e^{-i\alpha}\langle E_0,v_2\rangle-\frac{i}{2}\langle E_2,v_2\rangle\right),\\[2mm]
		\gamma_{12}(z)=2i\left(2e^{-i\alpha}\theta(\Im z)+\langle v_2,E_0(\cdot,{\bar{z}})\rangle+e^{-i\alpha}\langle E_0,v_1\rangle+\frac{i}{2}\langle E_2,v_1\rangle\right),\\[2mm]
		\gamma_{21}(z)=-2i\left(2e^{i\alpha}\theta(-\Im z)+e^{i\alpha}\langle v_1,E_0(\cdot,{\bar{z}})\rangle+\langle E_0,v_2\rangle-\frac{i}{2}\langle E_1,v_2\rangle	\right)	,\\[2mm]
		\gamma_{22}(z)=-2i\left(-\sign(\Im z)+\langle v_1,E_0(\cdot,{\bar{z}})\rangle-\langle E_0,v_1\rangle+\frac{i}{2}\langle E_1,v_1\rangle\right).\\[2mm]
	\end{array}
\end{equation}
The matrix $\Gamma(z)=\|\gamma_{jk}(z)\|$ has the property $\Gamma^*(\overline{z})=\Gamma(z),$ and $\gamma(z) =\det \Gamma$ has the property
$\overline{\gamma}(\overline{z})=\gamma(z)\neq 0$.
\end{theorem}
\begin{proof}
For the symmetric operator $A(v_1,v_2, \alpha)$,
%$A_{(v_1,v_2, \alpha)}$,
the construction of the Green's function is reduced to the equation
\begin{equation}\label{3.30}
\psi(x)=-\mathcal{E}_1(x,z)\psi_1-\mathcal{E}_2(x,z)\psi_2+\int\,\mathcal{G}_z(x,y;\alpha)h(y)dy,
\end{equation}
 where
\begin{equation}
\begin{array}{ll}
 E_0(x,z)=\mathcal{G}_z(x,-0;\alpha),&\quad E_k(x,z)=\int\, \mathcal{G}_z(x,y;\alpha)v_k(y)\,dy,\quad k=1,2,\\[2mm]
\mathcal{E}_1(x,z)=E_1+2iE_0, &\quad \mathcal{E}_2(x,z)=E_2-2ie^{-i\alpha}E_0,
\end{array}
 \end{equation}
and
\begin{equation}\label{3.33}
   \displaystyle  \psi_1=\psi(-0)+\frac{i}{2}\langle\psi,v_1\rangle,\quad\psi_2=\psi(+0)-\frac{i}{2}\langle\psi,v_2\rangle.
  \end{equation}
Indeed, the function $\psi(x)$ of $(\ref{3.30})$ satisfies the differential equation $(A(v_1,v_2,\alpha)-z_kI)\psi=h$ for any $h\in L_2$
%$(A_{(v_1,v_2,\alpha)}-z_kI)\psi=h$ for any $h\in L_2$
 and nonreal $z_k$ such that $\Im z_1 >0$ and $\Im z_2 <0$.
The equation $(\ref{3.30})$ has the following equivalent form:
\begin{equation}\label{3.31}
\psi(x)=E_0(x;z)[-2i\psi_1+2ie^{-i\alpha}\psi_2]-E_1(x,z)\psi_1-E_2(x,z)\psi_2+\int\,\mathcal{G}_z(x,y;\alpha)h(y)dy.
\end{equation}
From the equation $(\ref{3.31})$, using the fact that $\mathcal{G}_z(+0,y;\alpha)=e^{i\alpha}\mathcal{G}_z(-0,y;\alpha)$, we obtain
  \begin{equation}\label{3.32}
  \begin{array}{l}
   \psi(-0)=2\left(\psi_1-e^{-i\alpha}\psi_2\right)\theta(\Im z) -\langle v_1,E_0(\cdot,{\bar{z}})\rangle\psi_1-\langle v_2,E_0(\cdot,{\bar{z}})\rangle\psi_2+\langle h,E_0(\cdot,{\bar{z}})\rangle,\\[2mm]
   \psi(+0)=-2e^{i\alpha}\left(\psi_1-e^{-i\alpha}\psi_2\right)\theta(-\Im z)-e^{i\alpha}\langle v_1,E_0(\cdot,{\bar{z}})\rangle\psi_1-e^{i\alpha}\langle v_2,E_0(\cdot,{\bar{z}})\rangle\psi_2+e^{i\alpha}\langle h,E_0(\cdot,{\bar{z}})\rangle,\\
\langle \psi,v_k\rangle=-2i\left(\psi_1-e^{-i\alpha}\psi_2\right)\langle E_0,v_k\rangle-\psi_1\langle E_1,v_k\rangle-\psi_2\langle E_2,v_k\rangle+\langle h,E_k(\cdot,{\bar{z}})\rangle,\quad k=1,2.
    \end{array}
  \end{equation}
  If we substitute the values $\psi(-0)$, $\psi(+0)$, $\langle\psi,v_1\rangle$ and $\langle\psi,v_2\rangle$ from $(\ref{3.32})$ to $(\ref{3.33})$, we get the explicit form of the system for $\psi_1$ and $\psi_2$.
It has the following form:
    \begin{equation}\label{3.34}
       \begin{pmatrix}    a_{11} &  a_{12} \\  a_{21} &   a_{22}        \end{pmatrix} \begin{pmatrix} \psi_1\\ \psi_2    \end{pmatrix}=
    \begin{pmatrix} \langle h, \mathcal{E}_1(\cdot,{\bar{z}})\rangle\\ \langle h, \mathcal{E}_2(\cdot,{\bar{z}})\rangle \end{pmatrix},
    \end{equation}
where
\begin{equation}
\begin{array}{ll}
\gamma(z)=a_{11}a_{22}-a_{12}a_{21}=\gamma_{11}\gamma_{22}-\gamma_{12}\gamma_{21}\\
\gamma_{11}(z)=a_{22}(z)&\gamma_{12}(z)=-a_{12}(z)\\
\gamma_{21}(z)=-a_{21}(z)&\gamma_{22}(z)=a_{11}(z),
\end{array}
\end{equation}
and
$$
\begin{array}{l}
		a_{11}(z)=-2i\left(-\sign(\Im z)+\langle v_1,E_0(\cdot,{\bar{z}})\rangle-\langle E_0,v_1\rangle+\frac{i}{2}\langle E_1,v_1\rangle\right),\\[2mm]
		a_{12}(z)=-2i\left(2e^{-i\alpha}\theta(\Im z)+\langle v_2,E_0(\cdot,{\bar{z}})\rangle+e^{-i\alpha}\langle E_0,v_1\rangle+\frac{i}{2}\langle E_2,v_1\rangle\right),\\[2mm]
		a_{21}(z)=2i\left(2e^{i\alpha}\theta(-\Im z)+e^{i\alpha}\langle v_1,E_0(\cdot,{\bar{z}})\rangle+\langle E_0,v_2\rangle-\frac{i}{2}\langle E_1,v_2\rangle	\right),\\[2mm]
		a_{22}(z)=2i\left(\sign(\Im z)+e^{i\alpha}\langle v_2,E_0(\cdot,{\bar{z}})\rangle-e^{-i\alpha}\langle E_0,v_2\rangle-\frac{i}{2}\langle E_2,v_2\rangle\right)	.\\[2mm]
	\end{array}
$$
A determinant
    $$\gamma(z)=a_{11}a_{22}-a_{12}a_{21}\rightarrow -4\ , \ \textrm{for}\ \  |\Im z|\rightarrow \infty,
    $$
     since $$
      \begin{array}{ll}
   a_{11} (z)=2i\sign(\Im z) +o(1),& a_{12}(z)=-4ie^{-i\alpha}\theta(\Im z)+o(1),\\[2mm]
   a_{21}(z)=4ie^{i\alpha}\theta(-\Im z)+o(1),&  a_{22}(z)=2i\sign(\Im z) +o(1), \,\,|\Im z|\rightarrow \infty.
       \end{array}
       $$
Therefore, there exist nonreal $z_1$,\, $z_2$ such that $\Im z_1>0,$\, $\Im z_2<0$ and $|\Im z_k|$,\, $j=1,2$  large enough that $\gamma(z)\neq 0$ and the system $(\ref{eq3a})$ has a unique solution. This gives us the self-adjointness of the operator $A(v_1,v_2,\alpha)$.
% $A_{(v_1,v_2,\alpha)}$.
Moreover, we can show that $\bar a_{jk}(\bar z)=a_{kj}(z)$, then $\bar \gamma_{jk}(\bar z)=\gamma_{kj}(z)$. Therefore the matrix $\Gamma(z)=\|\gamma_{jk}(z)\|$ has the property $\Gamma^*(\overline{z})=\Gamma(z),$ and $\gamma(z) =\det \Gamma$ has the property
$\overline{\gamma}(\overline{z})=\gamma(z)$.
A substitution of $\psi_1$ and $\psi_2$ to the equation
$(\ref{3.30})$ gives
the explicit form of the Green's function $(\ref{3.31d})$.
\end{proof}

\begin{remark}
  If we substitute $\mathcal{E}_1$ and $\mathcal{E}_2$ from $(\ref{3.32d})$ to $(\ref{3.31d})$, we obtain the explicit form of Green's function that depends on $E_0,E_1, E_2$:
\begin{equation}
	\mathcal{G}_{z}(x,y;v_1,v_2, \alpha)=\mathcal{G}_z(x,y;\alpha)-\frac{1}{\gamma(z)}\sum_{j,k=0}^{2}\,E_j(x,z)c_{j,k}(z)\bar{E}_k(y,\bar{z}),
\end{equation}
where
\begin{equation}
E_0(x,z)=\mathcal{G}_z(x,-0;\alpha), \quad E_k(x,z)=\int\,\mathcal{G}_z(x,y;\alpha)v_k(y)\,dy,\ \ k=1,2
\end{equation}
and
\begin{equation}
\begin{array}{lll}c_{00}(z)=-2i\gamma_{22}+2ie^{i\alpha}\gamma_{12}+2i\gamma_{11}-2ie^{-i\alpha}\gamma_{21}& c_{01}(z)=\gamma_{22}+e^{-i\alpha}\gamma_{21} & c_{02}(z)=\gamma_{12}+e^{-i\alpha}\gamma_{11} \\
c_{10}(z)=-\gamma_{22}+e^{i\alpha}\gamma_{12}&c_{11}(z)=-\frac{i}{2}\gamma_{22} & c_{12}(z)=-\frac{i}{2}\gamma_{12} \\
c_{20}(z)=-e^{i\alpha}\gamma_{11}+\gamma_{21}&c_{21}(z)=\frac{i}{2}\gamma_{21} & c_{22}(z)=\frac{i}{2}\gamma_{11}.
\end{array}
\end{equation}
The values $\gamma_{jk}(z)$ are given in Theorem \ref{thm3.5}. The matrix  $||c_{jk}||_{jk=0}^{2}$ is  degenerate. Its first column is a linear combination of second and third ones:
$$\col(c_{00},c_{10},c_{20})=-2i\col(c_{01},c_{11},c_{21})+2ie^{i\alpha}\col(c_{02},c_{12},c_{22}).
$$
The equalities
 $c_{jk}(z)=\bar{c}_{kj}(\bar{z}),$\,\,$\gamma(z)=\bar{\gamma}(\bar{z})\neq 0$ are valid
 providing the self-adjoint property of the operator $A_{(v_1,v_2,\alpha)}:$
 $${\bar{\mathcal{G}}}_{\bar{z}}(y,x;v_1,v_2,\alpha)=\mathcal{G}_z(x,y;v_1,v_2,\alpha).
 $$
\end{remark}

\section{Spectral analysis of the momentum operator with a nonlocal potential}
An operator $A(v,\alpha)$ of the form (\ref{2.16}) with a nonlocal potential $v \in L_2(-\infty, +\infty)$ is self-adjoint in the space $L_2(-\infty, +\infty)$. Its resolvent
$R_z(v,\alpha)=(A(v,\alpha)-zI)^{-1}$ for nonreal $z$ is an integral operator with kernel $\mathcal{G}_z(x,y;v, \alpha)$ which is a perturbation of rank at most 2 of the Green's function of the free operator $A_0=i\frac{d}{dx}$ with a domain $W_2^1(-\infty, +\infty)$. The spectrum of the self-adjoint operator $A_0$ is absolutely continuous and is composed of the whole real axis. Therefore, the operator $A(v,\alpha)$, in addition to the absolutely continuous part of the spectrum, also has ordinary eigenvalues with eigenfunctions belonging to the space $L_2(-\infty, +\infty)$.
\begin{theorem}\label{4.1}
Let a self-adjoint operator $A(v,\alpha)$ of the form (\ref{2.16}) have a nonlocal potential $v \in L_2(-\infty, +\infty)$. Then the real number $\lambda$ is an eigenvalue of
operator $A(v,\alpha)$ if and only if the function
\begin{equation}\label{4.1}
   \psi_{\lambda}(x)=\theta(x)\left(-i\int_{x}^{\infty}\,e^{-i\lambda(x-y)}v(y)\,dy \right)+\theta(-x)\left(i\int_{-\infty}^{x}\,e^{-i\lambda(x-y)}v(y)\,dy\right)
\end{equation}
belongs to the space $L_2(R^1)$ and two conditions are satisfied
\begin{equation}\label{4.2}
\begin{array}{l}
    \psi_{\lambda}(-0) + e^{-i\alpha} \psi_{\lambda}(+0) =2,\\
     \psi_{\lambda}(-0)- e^{-i\alpha} \psi_{\lambda}(+0)=-i\langle \psi_{\lambda},v\rangle_{L_2}.
\end{array}
\end{equation}
If these conditions are satisfied, then the function $\psi_{\lambda}(x)$ is an eigenfunction of the operator $A(v,\alpha)$ with an eigenvalue $\lambda$
\begin{equation}\label{4.3}
A(v,\alpha) \psi_{\lambda}(x)=\lambda \psi_{\lambda}(x).
\end{equation}
  Each eigenvalue always has a multiplicity of one.
\end{theorem}
\begin{proof}
If $\lambda$ is an eigenvalue of the operator $A(v,\alpha)$, then its eigenfunction is satisfied (\ref{4.3}). Therefore, for $x\neq 0 $ the equation holds
\begin{equation}\label{4.4}
i\frac{d \psi_{\lambda}(x)}{dx}+v(x)\frac{1}{2} \Bigl( \psi_{\lambda}(-0)+e^{-i \alpha} \psi_{\lambda}(+0)\Bigr)=\lambda \psi_{\lambda}(x),\quad x\neq 0.
\end{equation}
The value $ \psi_{\lambda}(-0)+e^{-i\lambda} \psi_{\lambda}(+0)\neq0$ because then the nontrivial solution (\ref{4.4}) does not belong to the space $L_2 $. Since the eigenfunction $\psi_{\lambda}$ is determined up to a scalar non-zero factor, then, without a loss of generality, we can assume that the first condition in (\ref{4.2}) is satisfied.
%vstavka!!!
Then the solution (\ref{4.4}) takes the form (\ref{4.1}). This solution decreases as $|x|\rightarrow \infty$, but must belong to the space $L_2.$ In addition, since $\psi_{\lambda} \in \mathcal{D}(A(v,\alpha))$, then the second condition in (\ref{4.2}) holds.
Otherwise, if the conditions (\ref{4.1})--(\ref{4.2}) are satisfied, then the function $ \psi_{\lambda}(x)$ of the form (\ref{4.1}) is a solution (\ref{4.4}), i.e., it is an eigenfunction with an eigenvalue $\lambda$. Since the eigenfunction for the eigenvalue $\lambda$ takes the form (\ref{4.1}), then  the eigenvalue $\lambda$ has a multiplicity one.
\end{proof}
We will demonstrate the efficiency of the conditions in Theorem \ref{4.1} using the following examples.
\begin{example}\label{4.1}
Let the potential $v(x)$ in the operator $A(v,\alpha)$ be equal to the complex constant $k$ in the interval $x\in(0,1)$ and $v(x)\equiv0$ for $x\notin [0,1].$ Then the operator $A(v,\alpha)$ has an eigenvalue  $\lambda=0$  only with $k=2ie^{i\alpha}.$ In this case, the eigenvalue is unique.
\end{example}
\begin{proof}
It is easy to check all the conditions in Theorem \ref{4.1} with $\lambda=0$ and $k=2ie^{i\alpha}.$ The eigenfunction $\psi_0(x)=2e^{i\alpha}(1-x)$ with $ x\in(0,1)$ and $\psi_0(x)\equiv0$ with $x \notin [ 0,1].$ If we assume $\lambda \neq 0$  with chosen $k$, then the first condition in (\ref{4.2}) has the form $e^{-i\alpha} (2ie^{i\alpha}\frac{1-e^{i\alpha}}{\lambda})=2. $
The modulus of the left-hand side in this equality is equal to $\displaystyle 2\frac{\sin \frac{\lambda}{2}}{\frac{\lambda}{2}}$ and  is strictly less than 2 for $\lambda \neq 0$. Therefore, there are no other eigenvalues for the chosen potential, except $\lambda=0$ .
\end{proof}
\begin{example}\label{4.2}
Let the nonlocal potential be $v(x)=\theta(x)ke^{-(1+i\gamma)x}$, where $\gamma$ is a real number. If $k=e^{i\alpha}2i$, then  the number $\lambda=\gamma$ is the only eigenvalue with the eigenfunction $\psi_{\gamma}(x)=\theta(x)(-ike^{-x})$ of the self-adjoint operator $A(v,\alpha)$ with such
  nonlocal potential.
\end{example}
\begin{proof}
From Theorem \ref{4.1}, we have
$$\psi_{\lambda}=\theta(x)\Bigl(-i\int_{x}^{\infty}\,e^{-i\alpha(\lambda-y)}v(y)\,dy\Bigr)=\theta(x)\frac{1}{1+i(\gamma-\lambda)}(-iv(x)).
$$
 It is  easy to check that conditions (\ref{4.2}) are satisfied only  for $\lambda=\gamma$.
\end{proof}
\begin{example}\label{4.3}
%Let the nonlocal potential $v(x)=k_-$ for $x \in (-1,0),$ and $v(x)=k_+$ for $x \in (0,1),$ where $ k_+$ and $k_-$ are constants and $v(x)\equiv 0$ for $x\notin [-1,1].$
% Then $\lambda=0$ will be the eigenvalue of the operator $ A(v,\alpha)$ if and only if $k_-=i\cdot \gamma,$ \,\,$k_+=e^{i\alpha} (2i-i\gamma)$ where $\gamma$ is a real number. This eigenvalue is unique.
An operator $A(v,0)$ with a nonlocal potential $v(x)=2i\sign x\cdot e^{-|x|}$ is a self-adjoint operator in the space $L_2(-\infty,+\infty)$. This operator has only two simple (non-multiple) eigenvalues
$\lambda=1,$ \, $\lambda=-1$ and absolutely continuous  spectrum which is the entire real line.  The eigenfunctions for the specified eigenvalues can be represented as
\begin{equation}\label{}
  \psi_{\lambda}(x)=\frac{2}{1+\lambda^2}e^{-|x|}(1+i\lambda\sign x).
\end{equation}
\end{example}
\begin{proof}
The specified function $\psi_{\lambda}\in L_2(-\infty,+\infty)$ satisfies equation
$$
i\frac{d\psi_{\lambda}}{dx}+v(x)=\lambda\psi_{\lambda},\quad x\neq 0
$$
 and conditions
  $$\psi_{\lambda}(+0)=1+i\lambda,\quad \psi_{\lambda}(-0)=1-i\lambda,\quad \langle\psi_{\lambda},v\rangle=2 \quad\text{with}\,\,\lambda=1 \,\,\,\text{and}\,\,\,\lambda=-1.
  $$
  Therefore, the function $\psi_{\lambda}$ belongs to the domain  of the operator $A(v,0)$ and satisfies equality
  $$A(v,0) \psi_{\lambda}= \lambda  \psi_{\lambda}\quad\text{with}\,\,\,\,\lambda=1\,\,\, \,\text{and}\,\,\,\,\lambda=-1.
  $$
  Therefore, $ \psi_{\lambda}$ is an eigenfunction of the operator $A(v,0)$  with eigenvalues $\lambda=1,$ \, $\lambda=-1$.
   Since the condition
   $$\frac{1}{2}( \psi_{\lambda}(-0)+ \psi_{\lambda}(+0))=2
   $$
    is satisfied only for $\lambda=1$ and $\lambda=-1,$ then the operator  $A(v,0)$  has no other eigenvalues.
\end{proof}
%\begin{theorem}\label{4.2}
%Let the potential $v(x)=0$ for $x<0$ and its Fourier transform
%$$
%\tilde{v}(\lambda)=\int_{0}^{\infty}\,e^{i\lambda y}v(y)\,dy
%$$
  %tend to zero $|\tilde{v}(\lambda)| \rightarrow 0$ with  $| \lambda|\rightarrow \infty.$ Then  eigenvalues  can exist only on a bounded set where $|\tilde{v}(\lambda)|=2.$
  %for $\lambda \overline{\in} [-a,a]$ and there
 % can exist eigenvalues  only in the interval $[-a,a]$.
%\end{theorem}
%\begin{proof}
%For $|\lambda|>a$ the first condition of Theorem \ref{4.1} is not satisfied. Therefore, all eigenvalues lie in the interval $[-a,a]$. In Theorem \ref{4.1}, the eigenvalue $\lambda$ has only one eigenfunction of the %form (\ref{4.1}).
%\end{proof}
\section{Operators on a finite interval}

 Consider  the  first order differential expression  $A (v_1,v_2)$ on a finite interval $\Omega=(0,1)$  with nonlocal potentials $v_1$ and $v_2,$
 where $v_1$ and $v_2$    are complex-valued functions from the space~$ L_2 (\Omega)$.
 The action of the expression $A (v_1,v_2)$  on an arbitrary $\psi$ function from the Sobolev space $W_2^1(\Omega)$ is given by
\begin{equation}\label{21.1}
A(v_1,v_2)\psi(x)=i\frac{d\psi(x)}{dx}+v_1(x)\left[\psi(0)-\frac{i}{2}\langle \psi,v_1\rangle_{L_2}\right]+v_2(x)\left[\psi(1)+\frac{i}{2}\langle \psi,v_2\rangle_{L_2}\right].
\end{equation}

The equality $(\ref{21.1})$ defines a maximal operator in the space $ L_2 (\Omega)$, whose domain is the space $W_2^1 (\Omega) $. This operator will be denoted by  $A (v_1,v_2)$.

\begin{lemma}\label{5.1}
The operator  $A (v_1,v_2)$ satisfies the Green's formula, which for two arbitrary functions $\psi,\varphi\in W_2^1 (\Omega)$ has the form
\begin{equation}\label{21.2}
   \begin{array}{rr}
\langle A\psi,\varphi\rangle_{L_2}-\langle\psi,A\varphi\rangle_{L_2}=i\{\left[\psi(1)+i\langle\psi,v_2\rangle\right]\cdot\left[\varphi(1)+i\langle\varphi,v_2\rangle\right]^*-\\
-\left[\psi(0)-i\langle\psi,v_1\rangle\right]\cdot\left[\varphi(0)-i\langle\varphi,v_1\rangle\right]^*\}.
     \end{array}
 \end{equation}
\end{lemma}

\begin{proof}
Substituting  $(\ref{21.1})$ into the left-hand side of $(\ref{21.2})$, it shows the same as expanding the parentheses in the right-hand side of $(\ref{21.2})$.
\end{proof}

\begin{definition}\label{df5.2}
 For the operator  $A (v_1,v_2)$ and a real number $\alpha\in[0,2\pi),$  consider  the operator  $A (v_1,v_2,\alpha)$  in the space $ L
 _2 (\Omega)$ as a restriction of operator  $A (v_1,v_2)$ on all functions of space $W_2^1 (\Omega)$  that satisfy the boundary condition
\begin{equation}\label{21.3}
\psi(1)+i\langle\psi,v_2\rangle_{L_2}=e^{i\alpha}[\psi(0)-i\langle\psi,v_1\rangle_{L_2}].
\end{equation}
We will denote the operator   $A (0, 0,\alpha)$  by $A_{\alpha}$ in the case $v_1\equiv 0$ and $v_2\equiv 0.$
\end{definition}
\begin{lemma}\label{5.2}
The operator $A(v_1,v_2,\alpha)$ is symmetric in the space $L_2(\Omega).$
\end{lemma}
This is a simple corollary of Lemma \ref{5.1}.
\begin{remark}\label{rem5.1}
In Definition \ref{df5.2}  of operator $A(v_1,v_2,\alpha)$ with nonlocal potentials $v_1,\,v_2 \in L_2(0,1),$ the domain  of the operator is defined as the set of all functions $\psi \in W_2^1(0,1)$ satisfying the boundary condition
$$\psi(1)-e^{i\alpha}\psi(0)=-i\Bigl[\langle\psi,v_2\rangle+e^{i\alpha}\langle\psi,v_1\rangle\Bigr].
$$
 This condition can be equivalently represented by  the functionals
 $$\psi_1=\psi(0)-\frac{i}{2}\langle\psi,v_1\rangle,\quad \psi_2=\psi(1)+\frac{i}{2}\langle\psi,v_2\rangle
 $$
 included in the definition of the operator $A(v_1,v_2,\alpha)$  in the form
 \begin{equation}\label{}
   \psi(1)-e^{i\alpha}\psi(0)=2[\psi_2-e^{i\alpha}\psi_1].
 \end{equation}
\end{remark}
\begin{theorem}\label{2th1}
The operator $ A_{\alpha}$ is a self-adjoint operator in the space $L_2 (\Omega)$. The spectrum of this operator consists of eigenvalues, all numbers of the form
\begin{equation}\label{21.4}
\lambda_n=-\alpha+2\pi n,\quad \textrm{where}\quad n=0,\pm 1,\pm 2,\ldots,
\end{equation}
and the eigenfunctions are $\psi_n(x)=e^{-i\lambda_nx}.$
For any values of the complex number $z$  that differ from the indicated eigenvalues, there is a resolvent  $(A_{\alpha}-zI)^{-1}$, which is an integral operator with a kernel $g(x,y;z)$, that is,
\begin{equation}\label{21.5}
(A_{\alpha}-zI)^{-1}h(x)=\int_0^1g(x,y;z)h(y)dy.
\end{equation}
The kernel $g(x,y;z)$ is explicitly expressed as
\begin{equation}\label{21.6}
g(x,y;z)=e^{-iz(x-y)}[-i\theta(x-y)+\beta],\quad \beta=\frac{ie^{-iz}}{e^{-iz}-e^{i\alpha}}.
\end{equation}
\end{theorem}
\begin{proof}
Let us prove representations $(\ref{21.5})$, $(\ref{21.6})$.  For this, consider the boundary value problem
\begin{equation}\label{21.7}
i\frac{d\psi(x)}{dx}-z\psi(x)=h(x),\quad \psi(1)=e^{i\alpha}\psi(0).
\end{equation}
The general solution of equation $(\ref{21.7})$ has the form
\begin{equation}\label{21.8}
\psi(x)=Ce^{-izx}-i\int_0^xe^{-iz(x-y)}h(y)dy,
\end{equation}
where $C$ is an arbitrary constant.  We choose it so that the solution $(\ref{21.8})$ satisfies the boundary conditions $(\ref{21.7})$.  This leads to the following equation, where $C:$
\begin{equation}\label{21.9}
C[e^{-iz}-e^{i\alpha}]=ie^{-iz}\int_0^1e^{izy}h(y)dy.
\end{equation}
If $z\neq \lambda_n,$ then there is a unique solution of equation $(\ref{21.9})$.  Substituting this value of $C$ in $(\ref{21.8})$, we obtain $(\ref{21.5})$ and $(\ref{21.6})$. The self-adjointness of the operator $A_{\alpha }$ follows from its symmetry due to the Green's formula and the existence of bounded operators $(A_{\alpha}-zI)^{-1}$ for all non-real~$z$. Note that each eigenvalue $\lambda_n $ has its own eigenfunction
\begin{equation}\label{21.10}
\psi_n(x)=e^{-i\lambda_n x}.
\end{equation}
\end{proof}
\begin{remark}\label{5.2}
%page3
The free moment operator $L=i\frac{d}{dx}$ is self-adjoint on the space $L_2(-\infty,\infty)$ with the domain $\mathcal{D}(L)=W_2^1(-\infty,\infty).$ The resolvent $(L-zI)^{-1},$ if $\Im z\neq 0$, is the integral operator
$$(L-zI)^{-1}h=\int g_z(x-y)h(y)dy,
$$
where the kernel of the integral operator is Green's function
\begin{equation}
g_z(x)=i\sign(\Im z)\cdot\theta(-\Im z  x) e^{-izx}.
\end{equation}
The  Green's function $g(x,y;z)$ of the form $(\ref{21.6})$ is a rank one  perturbation of function $g_z:$
\begin{equation}
g(x,y;z)=g_z(x-y)+e^{-izx}(\beta-i\theta(\Im z))e^{izy}= g_z(x-y)+e(x,z)w(z,\alpha)\bar{e}(y,\bar{z}),
\end{equation}
where $$
\begin{array}{l}
\displaystyle e(x,z)=g_z(x)+g_z(x-1), \quad w(z,\alpha)=i\left[\frac{1}{e^{iz}-e^{-i\alpha}}\theta(\Im z)-\frac{1}{e^{-iz}-e^{i\alpha}}\theta(-\Im z)\right],\\[2mm]
 \bar{w}(\bar{z},\alpha)=w(z,\alpha).
\end{array}
$$
\end{remark}
\begin{remark}
From the explicit form $(\ref{21.6})$ of the resolvent kernel $g(x,y;z)$, we have the following equality
\begin{equation}\label{21.11}
g(1,y;z)=e^{i\alpha}g(0,y;z)=e^{i\alpha}\beta(z)e^{izy},\quad \textrm{where}\quad
\beta=\frac{ie^{-iz}}{e^{-iz}-e^{i\alpha}}
\end{equation}
\begin{equation}\label{21.12}
\overline{g}(y,x;z)=g(x,y;\overline{z}).
\end{equation}
Note that the equality
$(\ref{21.12})$
is equivalent to the equality of the operators
$\left[(A_{\alpha}-zI)^{-1}\right]^*=(A_{\alpha}-\overline{z}I)^{-1}$,
which is equivalent to a self-adjointness of the operator $A_{\alpha}$.
\end{remark}
\begin{theorem}\label{2th2}
Consider two operators with several potentials $A(v_1,v_2,\alpha)$ and $A(v,0,\alpha)$, where $v(x)=v_1(x)+e^{i\alpha}v_2(x)$.
The domains of definition of these operators are equal and  described by the appropriate boundary conditions
$(\ref{21.3})$
\begin{equation}\label{21.13}
\psi(1)+i\langle\psi,v_2\rangle=e^{i\alpha}[\psi(0)-i\langle\psi,v_1\rangle];\quad
\psi(1)=e^{i\alpha}[\psi(0)-i\langle\psi,v\rangle].
\end{equation}
The difference between these operators is bounded self-adjoint operator of rank $2$:
\begin{equation}\label{21.14}
[A(v_1,v_2,\alpha)-A(v,0,\alpha)]\psi=\frac{i}{2}e^{-i\alpha}v_1\langle\psi,v_2\rangle-\frac{i}{2}e^{i\alpha}v_2\langle\psi,v_1\rangle.
\end{equation}
\end{theorem}
\begin{proof}
Two boundary conditions $(\ref{21.13})$ for the specified connection
$v=v_1+e^{i\alpha}v_2$ coincide.
According to $(\ref{21.1})$ the following equality is true
\begin{equation}\label{21.15}
[A(v_1,v_2,\alpha)-A(v,0,\alpha)]\psi=v_1[\psi(0)-\frac{i}{2}\langle\psi,v_1\rangle]+v_2[\psi(1)+\frac{i}{2}\langle\psi,v_2\rangle]-v[\psi(0)-\frac{i}{2}\langle\psi,v\rangle].
\end{equation}
If we substitute $\psi(1)$ from
$(\ref{21.13})$
into this equality and take into account the explicit form of
$v=v_1+e^{i\alpha}v_2$
we get
$(\ref{21.14})$.
The right side of $(\ref{21.14})$
is the sum of two one-dimensional mutually adjoint operators. Therefore, it is a bounded self-adjoint operator of rank at most $2$.
\end{proof}
%\begin{rem}
%It follows from Theorem $\ref{2th2}$ that the proof of self-adjointness of operators of the form
%$A(v_1,v_2,\alpha)$
%reduces to proving the self-adjointness of operators of a special form $A(v,0,\alpha)$ with one potential, since the perturbation of an arbitrary self-adjoint operator by a bounded self-adjoint operator does not %change self-adjointness.
%\end{rem}
%\begin{theorem}\label{thm5.3}
%The operator $A_{(v_1,v_2,\alpha)}$ from Definition \ref{df5.2} with nonlocal potentials $v_1,v_2 \in L_2(0,1)$ is self--adjoint in $L_2(0,1)$.
 %For nonreal $z$ a resolvent of this operator is an integral operator with a kernel $g_z(x,y;v_1,v_2,\alpha)$:
%\begin{
%$$g_z(x,y;v_1,v_2,\alpha)=g_z(x,y,\alpha)+\frac{1}{\gamma(z)}\sum_{j,k=0}^{3}\,e_j(x,z) \gamma_{j,k}(z)\bar{e}_k(y,\bar{z}),
%$$
%where
%$$\displaystyle e_0(x,z)=g_z(x)+g_z(x-1),\quad e_k(x,z)=\int\limits_{0}^{1}\,g_z(x,y,\alpha)v_k(y)\,dy,\,\,\,k=1,2,
%$$
%and $\gamma_{jk}=...$
%\end{theorem}
%\begin{proof}
%This theorem can be proved  analogous to Theorem \ref{thm2.4}--\ref{thm3.5} and reduced to a proof of existence a bounded operator $(A-zI)^{-1}$ for nonreal $z$. This in turn reduces to solvability proof of equation:
%We solve the equation
  %$$\psi(x)=e_0(x,z)a(\alpha,z)[\psi_2-e^{i\alpha}\psi_1]-e_2(x,z)\psi_1-e_3(x,z)\psi_2+\int\limits_{0}^{1}\,g_z(x,y,\alpha)h(y)\,dy,
  %$$
 % where
 % $$\psi_1=\psi(0)-\frac{i}{2}\langle\psi,v_1\rangle,\quad \psi_2=\psi(1)+\frac{i}{2}\langle\psi,v_2\rangle.
%  $$
%\end{proof}
%page4
\begin{remark}\label{k5}
Due to Theorem $\ref{2th2}$ we have
$$A\left(\frac{1}{2}v,\frac{e^{-i\alpha}}{2}v, \alpha\right)-A(v,0,\alpha)\equiv 0.$$
We will denote the operator $A(\frac{1}{2}v,\frac{e^{-i\alpha}}{2}v, \alpha)$ by $A(v,\alpha). $ Its action is expressed as follows
\begin{equation}
	A(v,\alpha)\psi(x)= i\frac{d\psi(x)}{dx}+\frac{1}{2}v(x)\left[\psi(0)+e^ {-i\alpha} \psi(1)\right]
\end{equation}
with the boundary condition
\begin{equation}
	ie^{-i\alpha}\psi(1)-i\psi(0)=\langle\psi,v\rangle.
\end{equation}
\end{remark}

\begin{theorem}\label{thm5.3}
The operator $A(v_1,v_2,\alpha)$ from Definition  \ref{df5.2}
with nonlocal potentials $v_1,v_2 \in L_2(0,1)$ is self--adjoint in $L_2(0,1)$.
 For nonreal $z$,  the resolvent $(A(v_1,v_2,\alpha)-zI)^{-1}$ is an~integral operator with the kernel
%\begin{equation}
%g_z(x,y;v_1,v_2,\alpha)=g_z(x,y,\alpha)-\frac{1}{\gamma(z)}\sum_{j,k=0}^{2}\,E_j(x,z) c_{j,k}(z)\bar{E}_k(y,\bar{z}),
%\end{equation}
\begin{equation}\label{eq5a}
g_z(x,y;v_1,v_2,\alpha)=g_z(x,y;\alpha)-\frac{1}{\gamma}\sum\limits_{j,k=1}^{2}\mathcal{E}_j(x,z)\gamma_{jk}\bar{\mathcal{E}}_k(y,\bar{z}).
\end{equation}
where $g_z(x,y;\alpha):=g(x,y;z)$ is the Green's function from Theorem \ref{2th1},
\begin{equation}\label{eq4a}
\begin{array}{ll}
\displaystyle  E_0(x,z)=g_z(x,1;\alpha), &\quad E_k(x,z)=\int_{0}^{1}\,g_z(x,y;\alpha)v_k(y)\,dy,\ \ k=1,2,\\[2mm]
\mathcal{E}_1(x,z)=E_1-2ie^{i\alpha}E_0, &\quad \mathcal{E}_2(x,z)=E_2+2iE_0.
\end{array}
\end{equation}
and
%\begin{equation}\begin{array}{lll}c_{00}=4\gamma_{11}-4e^{i\alpha}\gamma_{12}-4e^{-i\alpha}\gamma_{21}+4\gamma_{22}& c_{01}=-2ie^{i\alpha}\gamma_{11}+2i\gamma_{21} & %c_{02}=-2ie^{i\alpha}\gamma_{12}+2i\gamma_{22} \\
%c_{10}=2ie^{-i\alpha}\gamma_{11}-2i\gamma_{12}&c_{11}=\gamma_{11} & c_{12}=\gamma_{12} \\
%c_{20}=2ie^{-i\alpha}\gamma_{21}-2i\gamma_{22}&c_{21}=\gamma_{21} & c_{22}=\gamma_{22}
%\end{array}
%\end{equation}
%with
%\begin{equation}
%\begin{array}{ll}
%\gamma(z)=\gamma_{11}\gamma_{22}-\gamma_{12}\gamma_{21}\\
%\gamma_{11}(z)=a_{22}(z)&\gamma_{12}(z)=-a_{12}(z)\\
%\gamma_{21}(z)=-a_{21}(z)&\gamma_{22}(z)=a_{11}(z)
%\end{array}
%\end{equation}
\begin{equation}\label{k6}
\begin{array}{l}
\gamma=\gamma_{11}\gamma_{22}-\gamma_{12}\gamma_{21},\\[2mm]
	\gamma_{11}(z)=-2i\left(1+2i\beta(z)-\langle E_0,v_2\rangle +\langle v_2,E_0(\cdot,{\bar{z}})\rangle+\frac{i}{2}\langle E_2,v_2\rangle\right)\\[2mm]
		\gamma_{12}(z)=-2i\left(2ie^{iz}\beta(z)+\langle E_0,v_1\rangle +e^{-i\alpha}\langle v_2,E_0(\cdot,{\bar{z}})\rangle-\frac{i}{2}\langle E_2,v_1\rangle\right),\\[2mm]
\gamma_{21}(z)=2i\left(-2ie^{i\alpha}\beta(z)+e^{i\alpha}\langle E_0,v_2\rangle +\langle v_1,E_0(\cdot,{\bar{z}})\rangle+\frac{i}{2}\langle E_1,v_2\rangle\right),\\[2mm]
	\gamma_{22}(z)=2i\left(1-2ie^{i\alpha}e^{iz}\beta(z)-e^{i\alpha}\langle E_0,v_1\rangle +e^{-i\alpha}\langle v_1,E_0(\cdot,{\bar{z}})\rangle-\frac{i}{2}\langle E_1,v_1\rangle\right).
	\end{array}
\end{equation}
The explicit form of the function $\beta(z)$ is given in $(\ref{21.6})$.
The matrix $\Gamma(z)=\|\gamma_{jk}(z)\|$ has the property $\Gamma^*(\overline{z})=\Gamma(z),$ and $\gamma(z) =\det \Gamma$ has the property
$\overline{\gamma}(\overline{z})=\gamma(z)\neq 0$.
\end{theorem}

\begin{proof}
For the symmetric operator $A(v_1,v_2, \alpha)$, the construction of the Green's function and~proof of self-adjointness is reduced to solve the problem
$$\left(A(v_1,v_2, \alpha)-z_kI\right)\psi=h,
$$
for any $h \in L_2(0,1)$ and nonreal $z_k$,\,\,$\Im z_1 >0,$\,\,$\Im z_2 <0$.
The solvability of this equation is~equivalent to solvability of the following equation:
\begin{equation}\label{k1}
i\psi'(x)-z_k\psi(x)=h(x)-v_1(x)\left[\psi(0)-\frac{i}{2}\langle \psi,v_1\rangle_{L_2}\right]-v_2(x)\left[\psi(1)+\frac{i}{2}\langle \psi,v_2\rangle_{L_2}\right]
\end{equation}
with boundary condition $(\ref{21.3})$.
The solution of the equation $(\ref{k1})$ has the following form
\begin{equation}\label{1a}
  \psi(x)=-E_0(x,z)w- E_1(x,z)\psi_1-E_2(x,z)\psi_2+\int_{0}^{1}\,g_z(x,y;\alpha)h(y)\,dy,
\end{equation}
where
  \begin{equation}\label{eq1a}
   \displaystyle  \psi_1=\psi(0)-\frac{i}{2}\langle\psi,v_1\rangle,\,\,\psi_2=\psi(1)+\frac{i}{2}\langle\psi,v_2\rangle
  \end{equation}
%$$E_0(x,z)=g_z(x,1,\alpha), \quad E_k(x,z)=\int_{0}^{1}\,g_z(x,y,\alpha)v_k(y)\,dy, \ \ k=1,2.
%$$
with the boundary condition
$$\psi(1)-e^{i\alpha}\psi(0)=-iw, \quad w=\langle\psi,v_2\rangle+e^{i\alpha}\langle\psi,v_1\rangle.
$$
%Moreover, it satisfies the differential equation we need. Therefore $(A-zI)\psi=h$.
According to Remark \ref{rem5.1}, we have the following equality:
$$\psi(1)-e^{i\alpha}\psi(0)=-iw=2(\psi_2-e^{i\alpha}\psi_1).
$$
Therefore,  the equation (\ref{1a}) can be represented as following:
\begin{equation}\label{eq2a}
  \begin{array}{l}
\displaystyle    \psi(x)=-E_0\cdot2i(\psi_2-e^{i\alpha}\psi_1)-E_1\psi_1 - E_2\psi_2+\int\limits_{0}^{1}\,g_z(x,y;\alpha)h(y)\,dy, \,\, \text{or} \\ [2mm]
 \displaystyle    \psi(x)=-\mathcal{E}_1(x,z)\psi_1-\mathcal{E}_2(x,z)\psi_2+\int\limits_{0}^{1}\,g_z(x,y,\alpha)h(y)\,dy.
  \end{array}
\end{equation}
%where
%\begin{equation}\label{eq4}
%\mathcal{E}_1(x,z)=E_1-2ie^{i\alpha}E_0, \quad \mathcal{E}_2(x,z)=E_2+2iE_0.
%\end{equation}
The system for $\psi_1$ and $\psi_2$ has the form (\ref{eq1a}),
%  \begin{equation}\label{eq1}
%   \displaystyle  \psi_1=\psi(0)-\frac{i}{2}\langle\psi,v_1\rangle,\,\,\psi_2=\psi(1)+\frac{i}{2}\langle\psi,v_2\rangle,
%  \end{equation}
   where substituting the values $\psi(0)$ and $\psi(1)$ for function $\psi(x)$ from $(\ref{eq2a})$ in the right--hand side, and the function from $(\ref{eq2a})$  in the scalar products $\langle\psi,v_1\rangle$, $\langle\psi,v_2\rangle$.
   The explicit form of system $(\ref{eq1a})$ has the following form:
    \begin{equation}\label{eq3a}
       \begin{pmatrix}    a_{11} &  a_{12} \\  a_{21} &   a_{22}        \end{pmatrix} \begin{pmatrix} \psi_1\\ \psi_2    \end{pmatrix}=
    \begin{pmatrix} \langle h, \mathcal{E}_1(\cdot,{\bar{z}})\rangle\\
\langle h, \mathcal{E}_2(\cdot,{\bar{z}})\rangle \end{pmatrix},
    \end{equation}
where
\begin{equation}
\begin{array}{ll}
\gamma(z)=a_{11}a_{22}-a_{12}a_{21}=\gamma_{11}\gamma_{22}-\gamma_{12}\gamma_{21}\\
\gamma_{11}(z)=a_{22}(z)&\gamma_{12}(z)=-a_{12}(z)\\
\gamma_{21}(z)=-a_{21}(z)&\gamma_{22}(z)=a_{11}(z)
\end{array}
\end{equation}
and
\begin{equation}
\begin{array}{l}
		a_{11}(z)=2i\left(1-2ie^{i\alpha}e^{iz}\beta(z)-e^{i\alpha}\langle E_0,v_1\rangle +e^{-i\alpha}\langle v_1,E_0(\cdot,{\bar{z}})\rangle-\frac{i}{2}\langle E_1,v_1\rangle\right),\\[2mm]
a_{12}(z)=2i\left(2ie^{iz}\beta(z)+\langle E_0,v_1\rangle +e^{-i\alpha}\langle v_2,E_0(\cdot,{\bar{z}})\rangle-\frac{i}{2}\langle E_2,v_1\rangle\right),\\[2mm]
a_{21}(z)=-2i\left(-2ie^{i\alpha}\beta(z)+e^{i\alpha}\langle E_0,v_2\rangle +\langle v_1,E_0(\cdot,{\bar{z}})\rangle+\frac{i}{2}\langle E_1,v_2\rangle\right),\\[2mm]
		a_{22}(z)=-2i\left(1+2i\beta(z)-\langle E_0,v_2\rangle +\langle v_2,E_0(\cdot,{\bar{z}})\rangle+\frac{i}{2}\langle E_2,v_2\rangle\right).\\[2mm]
	\end{array}
\end{equation}
A determinant
    $$\gamma(z)= a_{11} a_{22}-a_{12} a_{21}\rightarrow -4 \ , \ \textrm{for}\ \  |\Im z|\rightarrow \infty,
    $$
     since $$
      \begin{array}{l}
   a_{11} (z)=2i\left(1-2ie^{i\alpha}e^{iz}\beta(z)\right) +o(1),\\[2mm]
a_{12}(z)=2i\left(2ie^{iz}\beta(z)\right)+o(1),\\[2mm]
   a_{21}(z)=-2i\left(-2ie^{i\alpha}\beta(z)\right)+o(1),\\[2mm]
a_{22}(z)=-2i\left(1+2i\beta(z)\right) +o(1), \,\,|\Im z|\rightarrow \infty.
       \end{array}
       $$
%\begin{align}
%&\left[2i\left(1-2ie^{i\alpha}e^{iz}\beta(z)\right)\right]\left[-2i\left(1+2i\beta(z)\right)\right]-\left[2i\left(2ie^{iz}\beta(z)\right)\right]\left[-2i\left(-2ie^{i\alpha}\beta(z)\right)\right]\\
%&=4\left[1-2ie^{i\alpha}e^{iz}\beta(z)+2i\beta(z)\right]=-4.
%\end{align}
Therefore, there exist nonreal $z_1$, $z_2$ such that $\Im z_1>0$,\,\, $\Im z_2<0$ and $|\Im z_k|$,\, $j=1,2$   large enough that $\gamma(z)\neq 0$ and the system $(\ref{eq3a})$ has a unique solution.
This gives us the self-adjointness of the operator $A(v_1,v_2,\alpha)$.
We can show that $\bar a_{jk}(\bar z)=a_{kj}(z)$, then $\bar \gamma_{jk}(\bar z)=\gamma_{kj}(z)$. Therefore, the matrix $\Gamma(z)=\|\gamma_{jk}(z)\|$ has the property $\Gamma^*(\overline{z})=\Gamma(z),$ and $\gamma(z) =\det \Gamma$ has the property
$\overline{\gamma}(\overline{z})=\gamma(z)$.
Substituting  the values of $\psi_1$,\, $\psi_2$ in the right--hand side of $\psi(x)$ in the equation  $(\ref{eq2a}),$ %calculated from $(\ref{eq3a})$
 we obtain the representation of Green's function $(\ref{eq5a})$.
%\begin{equation}\label{eq5a}
%%\mathcal{G}_z(x,y,v_1,v_2,\alpha)=g_z(x,y,\alpha)-\frac{1}{\gamma}\sum\limits_{j,k=1}^{2}\mathcal{E}_j(x,z)\gamma_{jk}\bar{\mathcal{E}}_k(y,\bar{z}).
%\end{equation}
\end{proof}
\begin{remark}\label{k4}
  If we substitute $\mathcal{E}_1$ and $\mathcal{E}_2$ from $(\ref{eq4a})$ to $(\ref{eq5a})$, we can describe the Green's function explicitly using $E_0,E_1, E_2$:
\begin{equation}\label{k2}
	g_{z}(x,y;v_1,v_2, \alpha)=g_z(x,y;\alpha)-\frac{1}{\gamma}\sum_{j,k=0}^{2}\,E_j(x,z)c_{j,k}(z)\bar{E}_k(y,\bar{z}),
\end{equation}
where
\begin{equation}\label{k3}
E_0(x,z)=g_z(x,1;\alpha), \quad E_k(x,z)=\int_{0}^{1}\,g_z(x,y;\alpha)v_k(y)\,dy,\ \ k=1,2
\end{equation}
and
\begin{equation}
\begin{array}{lll}c_{00}=4\gamma_{11}-4e^{i\alpha}\gamma_{12}-4e^{-i\alpha}\gamma_{21}+4\gamma_{22}& c_{01}=-2ie^{i\alpha}\gamma_{11}+2i\gamma_{21} & c_{02}=-2ie^{i\alpha}\gamma_{12}+2i\gamma_{22} \\
c_{10}=2ie^{-i\alpha}\gamma_{11}-2i\gamma_{12}&c_{11}=\gamma_{11} & c_{12}=\gamma_{12} \\
c_{20}=2ie^{-i\alpha}\gamma_{21}-2i\gamma_{22}&c_{21}=\gamma_{21} & c_{22}=\gamma_{22}.
\end{array}
\end{equation}
The values $\gamma_{jk}(z)$ are given in Theorem \ref{thm5.3}. The matrix  $||c_{jk}||_{jk=0}^{2}$ is  degenerate. Its first column is a linear combination of second and third ones:
$$\col(c_{00},c_{10},c_{20})=2ie^{-i\alpha}\col(c_{01},c_{11},c_{21})-2i\col(c_{02},c_{12},c_{22}).
$$
The equalities
 $c_{jk}(z)=\bar{c}_{kj}(\bar{z}),$\,\,$\gamma(z)=\bar{\gamma}(\bar{z})\neq 0$ are valid
 providing the self-adjoint property of the operator $A(v_1,v_2,\alpha):$
% $A_{(v_1,v_2,\alpha)}:$
 $${\bar{g}}_{\bar{z}}(y,x;v_1,v_2,\alpha)=g_z(x,y;v_1,v_2,\alpha).
 $$
\end{remark}
\begin{remark}
If we substitute $v_1=\frac{1}{2}v$ and $v_2=\frac{e^{-i\alpha}}{2}v$ to $(\ref{eq5a})$--$(\ref{k6})$ from Theorem $\ref{thm5.3}$, we obtain the explicit form of the resolvent of the operator $A(v,\alpha)$ from Remark $\ref{k5}$.
Indeed,
	\begin{equation}
		(A(v,\alpha)-zI)^{-1}h(z)=\int\,g_{z}(x,y;v, \alpha)h(y) \,dy
	\end{equation}
with the kernel
\begin{equation}
\displaystyle g_z(x,y;v,\alpha)=g_z(x,y;\alpha)-\frac{1}{\det F}\sum\limits_{j ,k=1}^{2}\,e_j(x,z)f_{jk}(z)\bar{e}_k(y,\bar{z}),
\end{equation}
 where $e_1(x,z)=g_z(x,1;\alpha),$\,\,$\displaystyle e_2(x,z)=\int_0^1\, g_z(x,y ;\alpha)v(y)\,dy$
and
\begin{equation}
	\begin{array}{ll}
f_{11}(z)=\langle e_2,v\rangle, \\[2mm]
f_{12}(z)=-\left(e^{i\alpha}+\langle v,e_1(\cdot,\overline{z})\rangle\right), \\[2mm]
f_{21}(z)=-\left(e^{-i\alpha}+\langle e_1,v\rangle\right), \\[2mm]
f_{22}(z)=\displaystyle \frac{i}{2} \frac{e^{-iz}+e^{i\alpha}}{e^{-iz}-e^{i\alpha}},\\[2mm]
\det F=f_{11}(z) f_{22}(z)-f_{12}(z) f_{21}(z).
	\end{array}
\end{equation}
\end{remark}
%page8
\section{Spectral analysis of momentum operator on a finite interval}

The considered models of self-adjoint operators are exactly solvable models. For the operator $A(v,0,\alpha)$, the eigenvalue problem is to find nontrivial solutions of the boundary value problem
\begin{equation}\label{3.1}
i\frac{d\psi}{dx}+v(x)\left[\psi(0)-\frac{i}{2}\langle\psi,v\rangle_{L_2}\right]=\lambda\psi,
\end{equation}
\begin{equation}\label{3.2}
\psi(1)=e^{i\alpha}\left[\psi(0)-i\langle\psi,v\rangle_{L_2}\right].
\end{equation}
The general solution of the equation
$(\ref{3.1})$ has the following form
\begin{equation}\label{3.3}
\psi(x)=C_1e^{-i\lambda x}-C_2\cdot i\cdot\int\limits_0^x e^{-i\lambda(x-y)}v(y)dy,
\end{equation}
where the constant
$C_1$ is arbitrary, and the constant $C_2$ is related to the solution by the equality
\begin{equation}\label{3.4}
C_2=-\left[\psi(0)-\frac{i}{2}\langle\psi,v\rangle_{L_2}\right].
\end{equation}
The condition that the solution
$\psi(x)$
of the form
$(\ref{3.3})$
satisfies the boundary condition $(\ref{3.2})$ leads to the equation
\begin{equation}\label{3.5}
C_1(e^{-i\lambda}-e^{i\alpha})-C_2\cdot i\cdot e^{-i\lambda}\tilde v=-ie^{i\alpha}\langle\psi,v\rangle,
\end{equation}
where
\begin{equation}\label{tildev}
\tilde v=\int\limits_0^1e^{i\lambda y}v(y)dy.
\end{equation}
If we substitute the solution $\psi$ of the form
$(\ref{3.3})$ into
$\langle\psi, v\rangle$, we get
\begin{equation}\label{3.6}
\langle\psi,v\rangle=C_1\tilde v^*-C_2i\hat v,
\end{equation}
where
\begin{equation}\label{3.7}
\hat{v}=\int\limits_0^1\int\limits_0^x e^{-i\lambda(x-y)}v(y)\overline{v(x)}dydx.
\end{equation}
Excluding from
$(\ref{3.5})$ and $(\ref{3.6})$ the value
$\langle\psi,v\rangle$, we have
\begin{equation}\label{3.8}
C_1\left(e^{-i\lambda}-e^{i\alpha}+ie^{i\alpha}\tilde v^*\right)-i C_2\left(e^{-i\lambda}\tilde v+ie^{i\alpha}\hat v\right)=0.
\end{equation}
We get
\begin{equation}\label{3.9}
C_1(2i+\tilde v^*)+i C_2(2-\tilde v)=0,
\end{equation}
excluding
$\psi(0)$ and $\langle\psi,v\rangle$ from $(\ref{3.3})$, $(\ref{3.4})$  and $(\ref{3.6})$.
Thus, the constants
$C_1$ and $C_2$ satisfy the homogeneous system of equations
 $(\ref{3.8})$,\, $(\ref{3.9})$. The~condition of the existence of non-trivial solutions of this system is equivalent to  the determinant
$\chi(\lambda)$
of this system is equal to zero.

%page9
This determinant $\chi(\lambda)$ has the following explicit form
\begin{equation}\label{3.10}
\chi(\lambda)=i(2i+\tilde v^*)(e^{-i\lambda}+ie^{i\alpha}\hat v)+i(2-\hat v)(e^{-i\lambda}-e^{i\alpha}+ie^{i\alpha}\tilde v^*).
\end{equation}
Given above leads to the following theorem.
\begin{theorem}
  The real zeros of the function
$\chi(\lambda)$ \,(\ref{3.3}) and only them
are the eigenvalues of the operator
$A(v,0,\alpha)$.
The non-trivial solutions of the system
$(\ref{3.8})$--$(\ref{3.9})$
give an explicit form of the eigenfunctions according to the formula
$(\ref{3.3})$.
\end{theorem}

Let us consider in more detail a particular case of the operator $A(v,0,\alpha)$,
when the nonlocal potential is a constant
$v(x)\equiv V$,\,\,$\alpha=\pi$
and
$e^{i\alpha}=-1$, that is the operator  $A(V,0,\pi).$
Then, using explicit form
$(\ref{tildev})$ and $(\ref{3.7})$ for
$\tilde v$ and $\hat v$, the characteristic function
$\chi(\lambda)$ can be written in
the following form
\begin{equation}\label{3.11}
\chi(\lambda)=4+2\frac{e^{-i\lambda}-1}{\lambda}\Bigl[\lambda+\overline{V}-V+\frac{i}{2}|V|^2\Bigr].
\end{equation}
The equation
$\chi(\lambda)=0$
can be rewritten as follows
\begin{equation}\label{3.12}
\frac{e^{-i\lambda}-1}{\lambda}=\frac{-2}{\lambda-2iS(V)},\quad\textrm{where}\quad
S(V)=\Im V-\frac{1}{4}|V|^2.
\end{equation}
Since the equation for the eigenvalues
$\lambda$
includes only the quantity
$S(V)$
that depends on the potential~$V$, then different potentials $V$ can have the same eigenvalues. So the potentials
$V=4i$ and $V=0$ have the same values
$S(4i)=S(0)=0$ and generate the same spectrum
$(\ref{21.4})$:\,$\lambda_n=(2n-1)\pi.$
\begin{defin}
 We will call the quantity
$S(V)=\Im V-\frac{1}{4}|V|^2$ the spectral characteristic of the constant nonlocal potential
$V$.
\end{defin}
\begin{lem}
Spectral characteristic $S(V)$ is less than or equal to $1$.
\end{lem}
\begin{lem}
A general form of constant nonlocal  potentials with $S(V)=0$ is
$V=2\sin{2\varphi}+i4\sin^2{\varphi}$, where $\varphi\in[0,\pi]$.
\end{lem}

%page10
Passing to trigonometric functions  in the equation $(\ref{3.12})$, we get an equivalent equation for
$\displaystyle \xi=\frac{2}{\pi}\lambda$,
\begin{equation}\label{3.13}
\displaystyle F(\xi)\equiv\frac{\tan{\frac{\pi}{4}\xi}}{\frac{\pi}{4}\xi},\quad F(\xi)=\frac{1}{S(V)},
\end{equation}
where
$\displaystyle \lambda=\frac{\pi}{2}\xi$.
Since the function
$F(\xi)$ is
even for
$\xi$,
then the solution of the equation
$(\ref{3.13})$
will be symmetric with respect to the point
$\xi=0$. The analysis of the equation
$(\ref{3.13})$ proves lemma.
\begin{stw}
The equation $(\ref{3.13})$ has a solution
$\xi=0$ and therefore
$\lambda=0$
if and only if $V=2i$.
In this case
$\lambda=0$ has a multiplicity $2$.
The orthogonal eigenfunctions corresponding to the value $\lambda=0$ have the form
$\psi_1=1$, $\displaystyle \psi_2=x-\frac{1}{2}$.
\end{stw}
\begin{defin}
We will call the potential $V=2i$  resonant.
\end{defin}
\begin{theorem}
The spectrum of the self--adjoint operator $A(v,0,\pi)$ consists of eigenvalues $\lambda$ which are solutions of the equation:
\begin{equation}\label{6.15}
  \frac{\tan \frac{\lambda}{2}}{\frac{\lambda}{2}}=S^{-1}(v).
\end{equation}
If the resonant potential $V=2i$  generates the spectrum: a double eigenvalue
$\lambda=0$  and positive eigenvalues
\begin{equation*}
\lambda_n=(2n+1)\pi-\frac{2}{(2n+1)\pi}+O\Bigl(\frac{1}{n^3}\Bigr),\quad n=1,2,\ldots
\end{equation*}
with $S^{-1}(V)>1$, then all positive eigenvalues consist of $\lambda_0\in(0,\pi)$,\,\,$\lambda_n\in(2n\pi,(2n+1)\pi)$ having the form
\begin{equation*}
\lambda_n=(2n+1)\pi-\frac{2}{(2n+1)\pi S^{-1}(V)}+O\Bigl(\frac{1}{n^3}\Bigr),\quad n=1,2,\ldots
\end{equation*}
If $S^{-1}(V)<0$, then all positive eigenvalues have the following form:
\begin{equation*}
\lambda_n=(2n-1)\pi-\frac{2}{(2n-1)\pi |S^{-1}(V)|}+O\Bigl(\frac{1}{n^3}\Bigr),\quad n=1,2,\ldots
\end{equation*}
All negative eigenvalues are symmetric to the positive ones  having the forms $-\lambda_n$.
\end{theorem}
\begin{proof}
  It follows from the solution of  equation (\ref{6.15}).
The graph of $F(\xi)$ demonstrates these statements  (see Figure 1). The intersection of the graph of  function $y=F(\xi)$ (see (\ref{3.13})) with horizontal lines $y=S^{-1}(V)$ gives the points with the abscissas  $\displaystyle \xi_n=\frac{2}{\pi}\lambda_n.$
\end{proof}
\newpage
\begin{figure}[ht]

\centering
	
\includegraphics[width=0.8\textwidth]{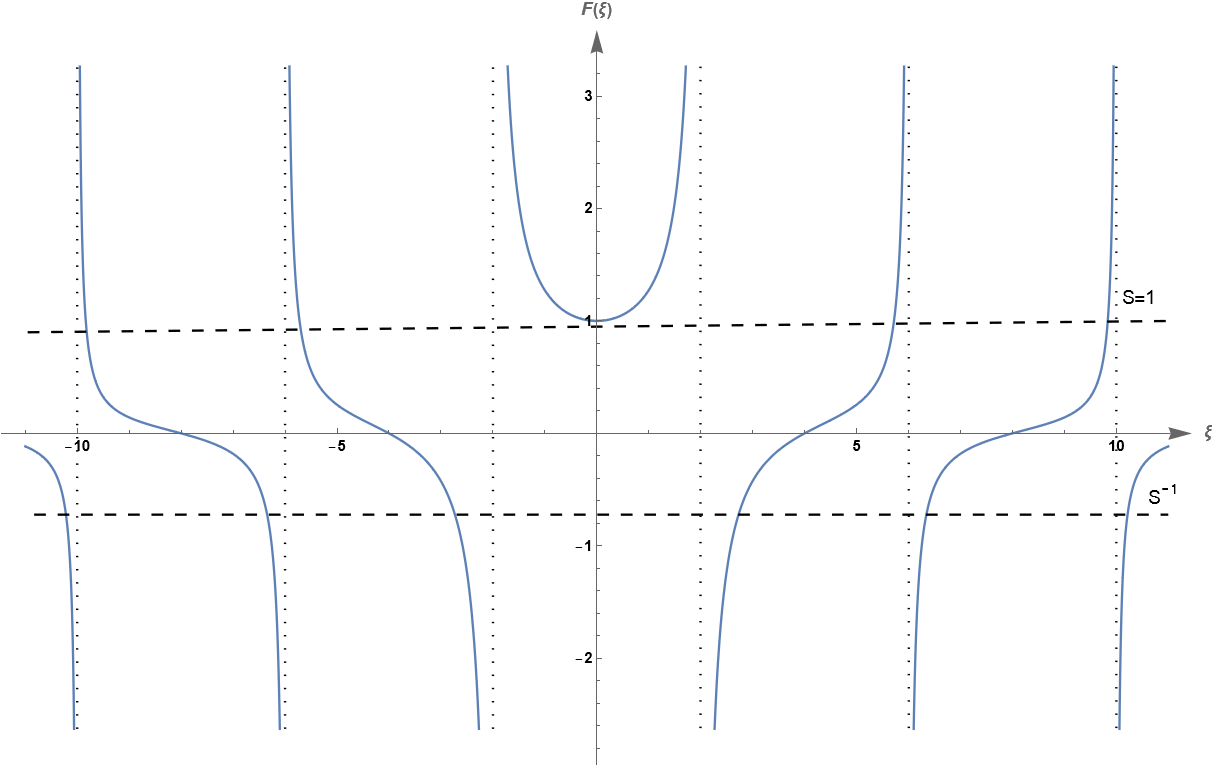}

\caption{Graph of F}
	
\label{fig1}
	
 \end{figure}

\newpage

\bigskip\noindent
{\bf Acknowledgments}
The authors wish to express their gratitude to Petru A. Cojuhari and  Leonid P. Nizhnik  for problems posed and many
useful discussions during the preparation of the manuscript. I.N. is grateful to the Simons Foundation for the financial support which has helped to accomplish  this manuscript.
%The authors thank  Leonid P. Nizhnik for  problems posed and numerous fruitful discussions.

\newpage
Kamila D\k{e}bowska\\
debowskamila@gmail.com\\
AGH University of Science and Technology\\
Faculty of Applied Mathematics\\
al. Mickiewicza 30, 30-059, Krakow, Poland\\ \ \\
Irina L. Nizhnik\\
irinanizh17@gmail.com\\
Institute of Mathematics of NASU\\
Tereshchenkivska Str.,3, 01 601 Kyiv, Ukraine
\end{document}